\documentclass{amsart}
\usepackage[USenglish]{babel}
\usepackage[T1]{fontenc}
\usepackage{mathtools,amsthm,amsfonts,amssymb}
\usepackage{enumitem}
\usepackage{xcolor}
\usepackage{cite}
\usepackage[colorlinks,allcolors=blue,backref=page]{hyperref}

\numberwithin{equation}{section}
\mathtoolsset{showonlyrefs}
\renewcommand*{\backrefalt}[4]{%
  \ifcase #1 %
  \or
    ↑~#2%
  \else
    ↑~#2%
  \fi}
\setlist[itemize,2]{label=$\circ$}

\newcommand{\sbar}[1]{\overline{\vphantom{1}#1}}

\newtheorem{theorem}{Theorem}[section]
\newtheorem{lemma}[theorem]{Lemma}
\newtheorem{proposition}[theorem]{Proposition}
\newtheorem{corollary}[theorem]{Corollary}

\theoremstyle{definition}
\newtheorem{definition}[theorem]{Definition}
\newtheorem{example}[theorem]{Example}

\theoremstyle{remark}
\newtheorem{remark}[theorem]{Remark}	

\title[Palindromic Poincar\'e polynomials in~\(B_n\)]{Palindromic Poincar\'e polynomials in~\(B_n\)}

\author{Gaston Burrull}
\address{(Gaston Burrull) \newline \indent Beijing International Center for Mathematical Research, Peking University, No.\@ 5 Yiheyuan Road, Haidian District, Beijing 100871, China}
\email{gaston(at)bicmr(dot)pku(dot)edu(dot)cn}
\keywords{Bruhat order, Poincar\'e polynomials, Lehmer codes, Kempf elements, rationally smooth Schubert varieties, Weyl groups of type B}
\subjclass[2020]{Primary 05E16; Secondary 20F55, 14M15, 06A07}

\begin{document}
\begin{abstract}
    We describe the set of palindromic Poincar\'e polynomials in~\(B_n\).
    As a key ingredient in this classification, we prove that \(A_n\) embeds into \(B_n\) as a lower interval in Bruhat order.
    Consequently, all Poincar\'e polynomials \(P_w(q)\) and all
    Kazhdan--Lusztig polynomials \(P_{x,y}(q)\) of~\(A_n\) occur
    in~\(B_n\).
\end{abstract}
\maketitle
\section{Introduction}

For a Coxeter group \((W,S)\) with Bruhat order \(\leq\) and length function \(\ell\), a basic problem is to understand the rank-generating functions of lower Bruhat intervals.
More precisely, for \(w\in W\), we consider the \emph{Poincar\'e polynomial}
\begin{equation*}
    P_w(q)=\sum_{u\leq w} q^{\ell(u)}.
\end{equation*}
When \(W\) is a Weyl group, \(P_w(q)\) is the Poincar\'e polynomial of the Schubert variety \(X_w\), with cohomological degrees divided by two.
Its palindromicity is equivalent to rational smoothness of \(X_w\) and reflects Poincar\'e duality in its rational cohomology.

The case when \(P_w(q)\) is non-palindromic remains much less understood than the palindromic case.
Some of the strongest structural results about factorizations of such polynomials come from \emph{Billey--Postnikov decompositions}~\cite{BP05,RS16,GG25}.

In contrast, when \(W\) is a finite Weyl group, palindromicity admits a precise factorization characterization.
Namely, \(P_w(q)\) is palindromic if and only if it factors as a product of \(q\)-integers
\begin{equation*}
    [m]_q:=1+q+\cdots+q^{m-1},
    \qquad m\in\mathbb Z_{\geq 1}.
\end{equation*}
This characterization follows from work of Akyildiz--Carrell, Gasharov, Billey, and Billey--Postnikov~\cite{AC12,Gas98,Bil98,BP05}; see Slofstra~\cite[Theorem~3.1]{Slo15} for a unified statement and proof.

We denote by \(\mathrm{RSm}(W)\) the set of \emph{rationally smooth} elements of~\(W\), that is, the elements \(w\) such that \(P_w(q)\) is palindromic.
In \(S_n\), rational smoothness is equivalent to smoothness~\cite{Deo85}, so we denote \(\mathrm{RSm}(S_n)\) by \(\mathrm{Sm}_n\).

Lehmer codes provide a combinatorial framework for many of the \(q\)-integer factorizations described above.
Following Bolognini--Sentinelli~\cite{BS25}, a \emph{Lehmer code} \(L\) of a finite Coxeter group \(W\) of rank \(r\), with exponents \(e_1,\ldots,e_r\), is a bijection
\begin{equation*}
    L\colon W\longrightarrow \prod_{i=1}^r\{0,1,\ldots,e_i\}
\end{equation*}
such that \(L^{-1}\) is order-preserving, where the codomain is equipped with the componentwise order \(\preceq\) and \(W\) with the Bruhat order \(\leq\).
We define \(\mathrm{Pr}(L)\subseteq W\) to be the set of \emph{\(L\)-principal elements}, that is, the set of \(w\in W\) such that, for every \(u\in W\),
\begin{equation*}
    u\leq w
    \quad\Longrightarrow\quad
    L(u)\preceq L(w).
\end{equation*}

For an \(L\)-principal element, the Lehmer code explicitly describes the elements of its lower Bruhat interval.
Indeed, the definition of \(L\)-principality, together with the fact that \(L^{-1}\) is order-preserving, shows that, for every \(w\in\mathrm{Pr}(L)\), the map \(L\) restricts to a bijection from the lower Bruhat interval \([\mathrm{id},w]\) onto the tuple interval
\begin{equation*}
    [0_r,L(w)]
    :=
    \{\mathbf a\in\mathbb Z_{\geq 0}^r:\mathbf a\preceq L(w)\},
    \qquad
    0_r=(0,\ldots,0).
\end{equation*}

The code of an \(L\)-principal element determines its Poincar\'e polynomial.
For a finite tuple \(\mathbf a=(a_1,\ldots,a_d)\) of nonnegative integers,
define its \emph{\(q\)-integer product} by
\begin{equation*}
    \Pi_q(\mathbf a):=\prod_{i=1}^d[a_i+1]_q,
\end{equation*}
which is the rank-generating function of the tuple interval
\([0_d,\mathbf a]\).
Every Lehmer code is rank-preserving, so
\begin{equation*}
    \ell(u)=\sum_{i=1}^r L(u)_i
\end{equation*}
for every \(u\in W\). Consequently, if \(w\in\mathrm{Pr}(L)\), then
\begin{equation*}
    P_w(q)
    =
    \sum_{\mathbf a\preceq L(w)}
    q^{a_1+\cdots+a_r}
    =
    \Pi_q(L(w)).
\end{equation*}

Our main result gives a complete description of the palindromic
Poincar\'e polynomials in \(B_n\) in terms of two explicit families of
\(q\)-integer products. We first recall the type~\(A\) description,
which provides one of these families.

\subsection{Palindromic Poincar\'e polynomials in \texorpdfstring{\(S_n\)}{Sn}}
\label{subsec:Sn}
We regard \(S_n\) as the symmetric group on \(\{1,\ldots,n\}\) and write its elements in the usual one-line notation \(w=w_1\cdots w_n\).

Let \(L_{S_n}\) denote the \emph{canonical Lehmer code for \(S_n\)}.
Under the standard identification \(S_n\cong A_{n-1}\), it agrees with
the code \(L_{A_{n-1}}\) of
Definition~\ref{def:lehmer code for An LAn}.
As observed in~\cite[Remark~7.15]{BS25}, a result of Abreu and Nigro~\cite[Theorem~1.6]{AN24} implies that all palindromic Poincar\'e polynomials of \(S_n\) are obtained from the \(L_{S_n}\)-principal elements.
More precisely,
\begin{equation}\label{eq:all palindromic for one Lehmer in Sn}
    \{P_w(q):w\in\mathrm{Sm}_n\}
    =
    \{P_w(q):w\in\mathrm{Pr}(L_{S_n})\}.
\end{equation}
Since different \(L_{S_n}\)-principal elements may have the same Poincar\'e polynomial, we next recall the nonredundant parametrization from~\cite[Section~7]{BS25}.

We record three elementary definitions.

Let \(F_n^{\uparrow}\) be the set of \emph{weakly increasing Fubini words} of length \(n\), that is,
\begin{equation*}
    F_n^{\uparrow}
    :=
    \left\{
        (a_1,\ldots,a_n):
        a_1=0
        \text{ and }
        a_i\in\{a_{i-1},a_{i-1}+1\}
        \text{ for }2\leq i\leq n
    \right\}.
\end{equation*}
In particular, \(\lvert F_n^{\uparrow}\rvert=2^{n-1}\).

Let \(U_n\subseteq S_n\) denote the set of \emph{unimodal permutations}, that is, the permutations avoiding the patterns \(213\) and \(312\).

For a collection \(\mathcal X\) of tuples, set
\begin{equation*}
    \Pi_q(\mathcal X):=
    \{\Pi_q(\mathbf a):\mathbf a\in\mathcal X\}.
\end{equation*}

Bolognini and Sentinelli show that each fiber of the map
\begin{equation*}
    w\longmapsto P_w(q),
    \qquad w\in\mathrm{Pr}(L_{S_n}),
\end{equation*}
contains a unique unimodal permutation. Moreover, the augmented-code map
\begin{equation*}
    \begin{aligned}
        U_n &\longrightarrow F_n^{\uparrow},\\
        w &\longmapsto (0,L_{S_n}(w)),
    \end{aligned}
\end{equation*}
is a bijection.
Since the initial zero contributes the factor \([1]_q=1\), for every
\(w\in U_n\) we have
\begin{equation*}
    \Pi_q\bigl((0,L_{S_n}(w))\bigr)
    =
    \Pi_q(L_{S_n}(w))
    =
    P_w(q).
\end{equation*}
Together, the preceding bijection and
Equation~\eqref{eq:all palindromic for one Lehmer in Sn} give
\begin{equation}\label{eq:palindromic-polynomials-Sn}
    \{P_w(q):w\in\mathrm{Sm}_n\}
    =
    \Pi_q(F_n^{\uparrow}).
\end{equation}
In particular, there are \(2^{n-1}\) distinct palindromic Poincar\'e polynomials in~\(S_n\).

\subsection{Palindromic Poincar\'e polynomials in~
\texorpdfstring{\(B_n\)}{Bn}}
We write the elements of \(B_n\) as signed permutations
\(w=w_1\cdots w_n\) in one-line notation. A bar indicates negation;
for example, the negative entries of
\(4\sbar6\sbar1\sbar2 3\sbar5\in B_6\) are
\(\sbar6,\sbar1,\sbar2,\) and \(\sbar5\).
Let \(L_{B_n}\) denote the \emph{canonical Lehmer code for \(B_n\)}
(Definition~\ref{def:lehmer code for Bn LBn}).

In \(B_n\), the situation is more subtle than in \(S_n\).
Although every palindromic Poincar\'e polynomial in \(B_n\) is still a
product of \(q\)-integers, the principal elements of \(L_{B_n}\)
account for only one part of the answer. Unlike in \(S_n\) and \(H_3\)~\cite[Example~6.9]{BS25}, this cannot, in general, be remedied
by choosing a different Lehmer code.

Indeed, the analogue of
\eqref{eq:all palindromic for one Lehmer in Sn} fails already in \(B_4\).
An exhaustive SageMath computation shows that, for every Lehmer code
\(L\) for \(B_4\), one has
\begin{equation}\label{eq: type-B-lehmer-insufficient}
    \left|\{P_w(q):w\in\mathrm{Pr}(L)\}\right|
    <
    \left|\{P_w(q):w\in\mathrm{RSm}(B_4)\}\right|.
\end{equation}

Our classification therefore consists of two families of polynomials:
the \emph{type~\(B\) family}, arising from \(L_{B_n}\), and the
\emph{type~\(A\) family}, transported from \(A_n\) through Stembridge's
\emph{bottom map}
(Theorem~\ref{thm: iso posets An to Bn intro}).

We first describe the type~\(B\) family.
A tuple \(\mathbf a=(a_1,\ldots,a_n)\in\mathbb Z^n\) is called a
\emph{reflected bounded-growth tuple}, or an \emph{rbg tuple}, if
\begin{equation*}
    0\leq a_i\leq 2i-1
    \quad\text{for }1\leq i\leq n,
\end{equation*}
and, for \(2\leq i\leq n\),
\begin{equation*}
    a_{i-1}<2i-3
    \quad\Longrightarrow\quad
    a_i\leq \min\{a_{i-1}+1,\,2i-3-a_{i-1}\}.
\end{equation*}
We denote the set of reflected bounded-growth \(n\)-tuples by \(\mathcal C_n\).
Note that \(F_n^{\uparrow}\subseteq\mathcal C_n\).

The type~\(B\) family is
\begin{equation*}
    \Pi_q(\mathcal C_n).
\end{equation*}
In Theorem~\ref{thm:equivalence of rbg codes}, we show that the tuples
in \(\mathcal C_n\) are precisely the codes of the \(L_{B_n}\)-principal elements.
We also characterize these elements by avoidance of eight signed patterns and show that there are \(C_n+C_{n+1}-1\) such elements, where \(C_n\) denotes the \(n\)th Catalan number.

Before introducing the second family, we fix the conventions for \(A_n\)
that we use throughout the paper.
We realize \(A_n\) as the symmetric group on
\(\{0,1,\ldots,n\}\) and write its elements in one-line notation as \(x=x_0x_1\cdots x_n\).
Under the standard identification \(A_n\cong S_{n+1}\), the element \(x_0x_1\cdots x_n\) corresponds to \((x_0+1)(x_1+1)\cdots(x_n+1)\).
For example, \(2310\in A_3\) corresponds to \(3421\in S_4\).
This zero-based convention is convenient for the one-line description
of the bottom map given below.

The type~\(A\) family is
\begin{equation*}
    \Pi_q(F_{n+1}^{\uparrow}),
\end{equation*}
which, by Section~\ref{subsec:Sn}, is the set of palindromic Poincar\'e polynomials in \(A_n\).

Our classification thus uses the principal elements of the two Lehmer
codes \(L_{A_n}\) and \(L_{B_n}\).
These principal elements are precisely the \emph{Kempf elements} in \(A_n\) and \(B_n\), respectively \cite[Propositions~4.3 and~4.9]{HL85}.
Lakshmibai introduced these elements to study the corresponding
\emph{Kempf varieties} and proved cohomology-vanishing theorems for
them~\cite{Lak76}.

The discussion above leads to our main theorem.
\begin{theorem}\label{thm:main-intro}
    The set of palindromic Poincar\'e polynomials in \(B_n\) is
    \begin{equation*}
        \Pi_q\bigl(
            \mathcal C_n\cup F_{n+1}^{\uparrow}
        \bigr).
    \end{equation*}
    Equivalently,
    \begin{equation*}
        \{P_w(q):w\in\mathrm{RSm}(B_n)\}
        =
        \left\{
            P_w(q):
            w\in\mathrm{Pr}(L_{A_n})\cup\mathrm{Pr}(L_{B_n})
        \right\}.
    \end{equation*}
\end{theorem}
The bottom map \(\operatorname{bot}\colon A_n\longrightarrow B_n\)
has the following description in one-line notation~\cite[\S4]{Ste97}.
For \(x=x_0x_1\cdots x_n\in A_n\), delete the first entry \(x_0\).
If \(x_0\neq0\), replace the remaining entry \(0\) by \(\overline{x_0}\).
The resulting signed permutation is \(\operatorname{bot}(x)\).
For example, \(\operatorname{bot}(0231)=231\) and \(\operatorname{bot}(2310)=31\overline2\).

The following theorem describes the interaction of the bottom map with Bruhat order and is a key ingredient in the proof of Theorem~\ref{thm:main-intro}.
\begin{theorem}
    \label{thm: iso posets An to Bn intro}
    The bottom map
    \begin{equation*}
        \operatorname{bot}\colon(A_n,\leq)\longrightarrow(B_n,\leq)
    \end{equation*}
    is a poset embedding with image \([\mathrm{id},b_0]\), where
    \begin{equation*}
        b_0=(n-1)(n-2)\cdots1\,\sbar n.
    \end{equation*}
    This interval consists precisely of the signed permutations in \(B_n\) with at most one negative entry.
\end{theorem}

\begin{corollary}\label{cor:intro An KL in Bn}
    The following statements hold:
    \begin{enumerate}[label=\textup{(\roman*)}]
        \item\label{cor:intro An KL in Bn-interval}
        For every \(x\in A_n\),
        \begin{equation*}
            [\mathrm{id},x]
            \cong
            [\mathrm{id},\operatorname{bot}(x)].
        \end{equation*}
        Thus every lower Bruhat interval in \(A_n\) is isomorphic as a poset
        to a lower Bruhat interval in \(B_n\).
        
        \item\label{cor:intro An KL in Bn-poincare}
        For every \(x\in A_n\),
        \begin{equation*}
            P_x(q)=P_{\operatorname{bot}(x)}(q).
        \end{equation*}
        In particular, every Poincar\'e polynomial in \(A_n\) occurs
        as a Poincar\'e polynomial in \(B_n\).

        \item\label{cor:intro An KL in Bn-KL}
        For every \(x,y\in A_n\) with \(x\leq y\),
        \begin{equation*}
            P_{x,y}(q)
            =
            P_{\operatorname{bot}(x),\operatorname{bot}(y)}(q).
        \end{equation*}
        In particular, every Kazhdan--Lusztig polynomial in \(A_n\)
        occurs in \(B_n\).
    \end{enumerate}
\end{corollary}

The first two assertions follow directly from Theorem~\ref{thm: iso posets An to Bn intro}.
The third assertion follows from the interval isomorphism in the first assertion, together with a result of Brenti, Caselli, and Marietti~\cite[Corollary~8.4]{BCM06}, which implies that the Kazhdan--Lusztig polynomials attached to corresponding subintervals of isomorphic lower Bruhat intervals agree.

\begin{example}[Palindromic Poincar\'e polynomials in \(B_3\)]\label{ex:palindromic-polynomials-B3}
    The type~\(B\) family is indexed by the following \(18\) tuples:
    \begin{equation*}
        \mathcal C_3
        =
        \left\{
        \begin{gathered}
            000,\ 001,\ 010,\ 011,\ 012,\\
            100,\ 101,\ 110,\ 111,\ 112,\ 120,\ 121,\
            130,\ 131,\ 132,\ 133,\ 134,\ 135
        \end{gathered}
        \right\}.
    \end{equation*}
    The following eight tuples index the type~\(A\) family:
    \begin{equation*}
        F_4^{\uparrow}
        =
        \{
            0000,\ 0001,\ 0011,\ 0012,\
            0111,\ 0112,\ 0122,\ 0123
        \}.
    \end{equation*}
    
    By Theorem~\ref{thm:main-intro}, the set of palindromic Poincar\'e
    polynomials in \(B_3\) is
    \(\Pi_q(\mathcal C_3\cup F_4^{\uparrow})\), whose \(13\) elements,
    ordered by degree, are
    \begin{equation*}
        \begin{gathered}
            1,\quad
            [2]_q,\quad
            [2]_q^2,\quad
            [2]_q[3]_q,\quad
            [2]_q^3,\quad
            [2]_q[4]_q,\quad
            [2]_q^2[3]_q,\\
            [2]_q^2[4]_q,\quad
            [2]_q[3]_q^2,\quad
            [2]_q[3]_q[4]_q,\quad
            [2]_q[4]_q^2,\quad
            [2]_q[4]_q[5]_q,\quad
            [2]_q[4]_q[6]_q.
        \end{gathered}
    \end{equation*}
    
    The restriction of \(\Pi_q\) to \(\mathcal C_3\) is not injective.
    For example,
    \begin{equation*}
        \Pi_q(101)
        =
        \Pi_q(110)
        =
        [2]_q^2
        =
        1+2q+q^2.
    \end{equation*}
    
    The type~\(A\) family contributes exactly one additional polynomial:
    \begin{equation*}
        \Pi_q(F_4^{\uparrow})
        \setminus
        \Pi_q(\mathcal C_3)
        =
        \bigl\{[2]_q[3]_q^2\bigr\}.
    \end{equation*}
    This polynomial is nevertheless realized in \(B_3\) through the bottom map.
    Indeed, for the unimodal permutation \(2310\in A_3\), we have
    \begin{equation*}
        (0,L_{A_3}(2310))=0122\in F_4^{\uparrow}
        \qquad\text{and}\qquad
        \operatorname{bot}(2310)=31\overline{2}\in B_3.
    \end{equation*}
    Since the bottom map preserves Poincar\'e polynomials,
    \begin{equation*}
        P_{31\overline{2}}(q)
        =
        P_{2310}(q)
        =
        \Pi_q(0122)
        =
        [2]_q[3]_q^2.
    \end{equation*}
    The tuple \(122\) would yield the same \(q\)-integer product, but \(122\notin\mathcal C_3\).
\end{example}

\subsection{Perspectives}

The results and techniques developed in this paper suggest three directions
for further research.

\begin{enumerate}
    \item \emph{Transport through the bottom map.}
    Beyond its role in our classification, the bottom map provides a
    mechanism for transporting further results from \(A_n\) to \(B_n\).
    Since it preserves Poincar\'e polynomials,
    \begin{equation*}
        \operatorname{bot}\bigl(\mathrm{RSm}(A_n)\bigr)
        =
        \operatorname{bot}(A_n)
        \cap
        \mathrm{RSm}(B_n).
    \end{equation*}
    For \(n\geq2\), the image \(\operatorname{bot}(A_n)\) is not a subgroup of \(B_n\), although it properly contains the standard parabolic subgroup \(S_n\subseteq B_n\) and retains the full Bruhat-order structure of \(A_n\).
    It can therefore encode information that is not visible through parabolic restriction alone.
    
    One concrete direction is to translate pattern avoidance in \(A_n\) into signed pattern avoidance in \(B_n\).
    Combined with the preservation of lower Bruhat intervals and Kazhdan--Lusztig polynomials established in Corollary~\ref{cor:intro An KL in Bn}, such a translation may yield type~\(B\) analogues of type~\(A\) results on smoothness, singularities, and Kazhdan--Lusztig polynomials.

    \item \emph{Palindromic Poincar\'e polynomials in \(D_n\).}
    As in type~\(B\), the palindromic Poincar\'e polynomials in \(D_n\)
    cannot, in general, be obtained from the principal elements of a single
    Lehmer code. Indeed, an exhaustive SageMath computation shows that,
    for every Lehmer code \(L\) for \(D_5\), one has
    \begin{equation*}
        \left|\{P_w(q):w\in\mathrm{Pr}(L)\}\right|
        <
        \left|\{P_w(q):w\in\mathrm{RSm}(D_5)\}\right|.
    \end{equation*}
    
    In \(B_n\), we overcome this obstruction by using the bottom map: the missing polynomials are transported from \(A_n\) through a lower-interval embedding
    \begin{equation*}
        (A_n,\leq)\hookrightarrow(B_n,\leq).
    \end{equation*}
    The same mechanism is unavailable in the ranks examined in type~\(D\).
    Indeed, the computations in Remark~\ref{rem:other-types-search} show that, for \(4\leq n\leq6\), no lower Bruhat interval in \(D_n\) is isomorphic to \((A_n,\leq)\).
    Thus, a classification in type~\(D\) requires a different method to account for the polynomials not produced by the principal elements of a chosen Lehmer code.

    A natural first step is to understand the principal part of such a classification.
    Lakshmibai gives an explicit factorization of the elements of \(D_n\), together with a condition \((K)\) defining the corresponding Kempf elements
    \cite[Proposition~D.1 and p.~339]{Lak76}.
    It would be interesting to determine whether these Kempf elements are precisely the principal elements of the Lehmer code \(L_{D_n}\) introduced in~\cite[Section~5.4]{BS25}, or whether they are the principal elements of a different Lehmer code for \(D_n\).
    Translating condition~\((K)\) into the corresponding code coordinates may yield a type~\(D\) analogue of the rbg description obtained here for \(B_n\).
    
    Comparing reduced words across different Coxeter groups may help account for the remaining polynomials.
    Stembridge studied relations among reduced words in types \(A\), \(B\), and \(D\)~\cite{Ste97}, from which the bottom map arises.
    
    \item \emph{Enumeration of palindromic Poincar\'e polynomials in \(B_n\).}
    Theorem~\ref{thm:main-intro} reduces the enumeration problem to understanding the map
    \begin{equation*}
        \Pi_q\colon
        \mathcal C_n\cup F_{n+1}^{\uparrow}
        \longrightarrow
        \mathbb Z_{\geq 0}[q].
    \end{equation*}
    The restriction of \(\Pi_q\) to \(F_{n+1}^{\uparrow}\) is injective, so
    \(F_{n+1}^{\uparrow}\) gives a nonredundant parametrization of the type~\(A\) family.
    In contrast, its restriction to \(\mathcal C_n\) is not injective, and the images of the two restrictions overlap.
    Understanding the fibers of \(\Pi_q\) on \(\mathcal C_n\) and the intersection
    \begin{equation*}
        \Pi_q(\mathcal C_n)
        \cap
        \Pi_q(F_{n+1}^{\uparrow})
    \end{equation*}
    may lead to a direct enumeration of the distinct palindromic
    Poincar\'e polynomials in \(B_n\).
\end{enumerate}
\subsection*{Acknowledgments}

This paper was made possible by many insightful discussions with Paolo Sentinelli, to whom I am deeply grateful.
I also thank him for his comments on an earlier version, which substantially improved the paper.
I thank Shixuan Zeng for suggesting the argument showing that avoidance of the eight patterns in~\eqref{eq: 8 bad patterns} implies the rbg condition.
I also thank Davide Bolognini for locating a copy of Lakshmibai's paper~\cite{Lak76}.
\section{Canonical words, Lehmer codes, and the bottom map}\label{section: canonical words}

We follow Stembridge's treatment of canonical words in~\(A_n\) and~\(B_n\) and of the bottom map~\cite{Ste97}, using right cosets in place of left cosets.
We also collect several other external results.

Let \(I\) be a finite index set of size \(r\) and let \(I^*\) be the free monoid on \(I\), i.e., the set of all words in the alphabet \(I\).  

Let \((W,S)\) be a rank \(r\) Coxeter system with generating set of simple reflections $S=\{s_i : i\in I\}$ and Bruhat order~$\leq$.
For a word \(\mathbf{i}=i_1i_2\cdots i_k\in I^*\), we define its
\emph{evaluation in \(W\)} by
\begin{equation*}
    \mathbf{s}(\mathbf{i})
    :=
    s_{i_1}s_{i_2}\cdots s_{i_k}\in W.
\end{equation*}
For \(i,j\in I\), let
\(m(i,j)\in\mathbb Z_{\geq1}\cup\{\infty\}\) denote the order of
\(s_is_j\), so that \(m(i,j)=1\) if and only if \(i=j\).
The matrix
\begin{equation*}
    (m(i,j))_{i,j\in I}
\end{equation*}
is the \emph{Coxeter matrix of \((W,S)\)}.

We now record the definitions of order on tuples, Lehmer codes, and principal elements.
\begin{definition}\label{def:tuple_order}
    We equip \(\mathbb Z_{\geq 0}^r\) with the componentwise order
    \(\preceq\), defined by
    \begin{equation*}
        \mathbf{x}=(x_1,\ldots,x_r)
        \preceq
        \mathbf{y}=(y_1,\ldots,y_r)
        \quad\Longleftrightarrow\quad
        x_i\leq y_i\text{ for all }1\leq i\leq r.
    \end{equation*}
\end{definition}

\begin{definition}\label{def:Lehmer_code and principal elements}
    Suppose that \(W\) is finite, with exponents
    \(e_1,\ldots,e_r\). A bijection
    \begin{equation*}
        L\colon W\longrightarrow
        \prod_{i=1}^{r}\{0,1,\ldots,e_i\}
    \end{equation*}
    is a \emph{Lehmer code for \(W\)} if its inverse is order-preserving;
    that is, for every
    \(\mathbf{x},\mathbf{y}\in
    \prod_{i=1}^{r}\{0,1,\ldots,e_i\}\), we have
    \begin{equation*}
        \mathbf{x}\preceq\mathbf{y}
        \quad\Longrightarrow\quad
        L^{-1}(\mathbf{x})\leq L^{-1}(\mathbf{y}).
    \end{equation*}

    We say that \(w\in W\) is \emph{\(L\)-principal} if, for every \(u\in W\),
    \begin{equation*}
        u\leq w
        \quad\Longrightarrow\quad
        L(u)\preceq L(w).
    \end{equation*}
    We denote by \(\mathrm{Pr}(L)\) the set of \(L\)-principal elements.
\end{definition}

We recall the terminology for reduced words needed to define canonical words.
For an element $w \in W$, its \emph{length $\ell(w)$} is the smallest nonnegative integer $k$ such that $w$ can be written as a product of~$k$ generators from $S$.

\begin{definition}\label{def:set of reduced words of w and W}
    We denote by \(\mathcal{R}(w)\) the set of \emph{reduced words
    of~\(w\)}, that is, the set of words
    \(\mathbf{i}=i_1i_2\cdots i_k\in I^*\) such that
    \[
        w=s_{i_1}s_{i_2}\cdots s_{i_k}
        \qquad\text{and}\qquad
        k=\ell(w).
    \]
    The set of \emph{reduced words in~\(W\)} is
    \[
        \mathcal{R}(W):=\bigcup_{w\in W}\mathcal{R}(w).
    \]
\end{definition}

\begin{definition}\label{def:subwords substrings and prefixes}
    A nonempty \emph{subword} of a word
    \(\mathbf{i}=i_1i_2\cdots i_k\in I^*\) is a word
    \(\mathbf{j}=i_{j_1}i_{j_2}\cdots i_{j_q}\in I^*\) obtained by
    choosing indices
    \[
        1\leq j_1<j_2<\cdots<j_q\leq k.
    \]
    If \(j_1,j_2,\ldots,j_q\) are consecutive integers, we call
    \(\mathbf{j}\) a \emph{substring} of~\(\mathbf{i}\). Such a
    substring is a \emph{prefix} of~\(\mathbf{i}\) if \(j_1=1\).
    We also regard the empty word as a subword, substring, and prefix
    of every word.
\end{definition}

\begin{lemma}[Subword Property {\cite[Theorem~2.2.2]{BB05}}]
\label{lem:Subword Property}
    Let \(u,w\in W\), and let \(\mathbf{i}\) be a reduced word of~\(w\).
    Then
    \begin{equation*}
        u\leq w
        \quad\Longleftrightarrow\quad
        \text{there exists a subword \(\mathbf{j}\) of~\(\mathbf{i}\)
        such that \(\mathbf{j}\in\mathcal R(u)\)}.
    \end{equation*}
\end{lemma}

We fix a total order on \(I\), so that $I^*$ inherits the lexicographic order.
This induces a maximal chain 
\begin{equation*}
    I_0\subset I_1\subset \cdots \subset I_{r}
\end{equation*}
of subsets of~\(I\), where \(I_0:=\emptyset\) and, for \(1\leq k\leq r\), \(i_k\) is the
minimal element of \(I\setminus I_{k-1}\) and
\[
    I_k:=I_{k-1}\cup\{i_k\}.
\]
For \(0\leq k\leq r\), let
\begin{equation*}
    W_k:=\langle s_i:i\in I_k\rangle.
\end{equation*}
Thus, \(W_0=\{\mathrm{id}\}\) and \(W_r=W\).

\begin{definition}\label{def:canonical-decomposition}
    For \(w\in W\), the \emph{canonical decomposition of \(w\)} is the
    unique factorization
    \begin{equation*}
        w=w^{(1)}w^{(2)}\cdots w^{(r)},
    \end{equation*}
    where, for \(1\leq k\leq r\), the \emph{\(k\)-th canonical factor}
    \(w^{(k)}\) is a minimal representative of a right coset in
    \(W_{k-1}\backslash W_k\). The existence and uniqueness of this
    factorization follow from~\cite[Corollary~2.4.6]{BB05}. Moreover,
    \begin{equation*}
        \ell(w)
        =
        \ell\bigl(w^{(1)}\bigr)
        +\ell\bigl(w^{(2)}\bigr)
        +\cdots
        +\ell\bigl(w^{(r)}\bigr).
    \end{equation*}
\end{definition}

\begin{definition}\label{def:canonical-words}
    Let
    \begin{equation*}
        w=w^{(1)}w^{(2)}\cdots w^{(r)}
    \end{equation*}
    be the canonical decomposition of \(w\in W\). The
    \emph{canonical word of \(w\)} is
    \begin{equation*}
        \mathbf c(w)
        :=
        \mathbf i_1\mathbf i_2\cdots\mathbf i_r,
    \end{equation*}
    where, for \(1\leq k\leq r\), \(\mathbf i_k\) is the lexicographically minimal element of \(\mathcal R(w^{(k)})\).
\end{definition}

For the subgroup chains used below in types \(A\) and \(B\), every minimal coset representative has a unique reduced word. Since every canonical factor is such a representative, lexicographic minimality is not needed to define canonical words.
For the subgroup chain used by Stembridge in type \(D\), some minimal coset representatives have two reduced words~\cite[Part~II]{Ste97}.

To compare reduced words in~\(A_n\) and~\(B_n\), we introduce a second Coxeter system \((W',S')\) also indexed by \(I\), with simple generators $S'=\{s'_i:i\in I\}$, Coxeter matrix \((m'(i,j))_{i,j\in I}\), and Bruhat order~\(\leq\).

For a word \(\mathbf{i}=i_1i_2\cdots i_k\in I^*\), we define its
\emph{evaluation in \(W'\)} by
\begin{equation*}
    \mathbf{s}'(\mathbf{i})
    :=
    s'_{i_1}s'_{i_2}\cdots s'_{i_k}
    \in W'.
\end{equation*}
For \(w\in W'\), we denote its canonical word by
\(\mathbf{c}'(w)\in I^*\).
We use the notation \(\mathcal R(w)\) for the set of reduced words of \(w\) in either Coxeter system, with the ambient group understood from context.
We denote by \(\mathcal R(W')\subseteq I^*\) the set of words reduced in \(W'\).
\begin{definition}\label{def:dominance between Coxeter systems}
    Following~\cite[Section~1.5]{Ste97}, we say that \(W\)
    \emph{dominates} \(W'\) if \(m(i,j)\geq m'(i,j)\) for all \(i,j\in I\).
\end{definition}

We now specialize to the classical Weyl groups \(W'=A_n\) and \(W=B_n\),
which are both of rank \(n\) (so \(r=n\)) and have the common index set
\begin{equation*}
    I=\{0,1,\ldots,n-1\}.
\end{equation*}
Their Coxeter presentations, recalled below, show that \(B_n\) dominates
\(A_n\).
We use this relation below to compare reduced words in \(A_n\) and
\(B_n\), and in the next section to prove
Theorem~\ref{thm: iso posets An to Bn intro}.

We equip \(I\) with its natural order, which determines canonical words \(\mathbf c'\) and \(\mathbf c\) for \(A_n\) and \(B_n\), respectively.
Up to reversing the indexing of the simple generators, the canonical decompositions in types \(A\) and \(B\) appear in~\cite[Propositions~A.1 and~B.1]{Lak76}, respectively.
The maps obtained by recording the lengths of their canonical factors are referred to as \emph{codes} in~\cite[Sections~3.1 and~4.1]{GK97} for types \(A\) and \(B\), respectively, up to our indexing convention in type~\(A\).
We refer to these maps as the canonical Lehmer codes \(L_{A_n}\) and \(L_{B_n}\); see Definitions~\ref{def:lehmer code for An LAn} and~\ref{def:lehmer code for Bn LBn}.

We now introduce notation for the consecutive strings that occur in these canonical factors.
\begin{definition}\label{def:consecutive_word_strings}
    For $i,j\in I$ with $0\leq i\le j$, we define
    \begin{itemize}
        \item $\mathbf{i}_{[i,j]}$ as the word obtained by the increasing sequence from \(i\) to \(j\), that is,
        \begin{equation*}
            \mathbf{i}_{[i,j]} = i\cdot(i+1)\cdots (j-1)\cdot j \in I^*.
        \end{equation*}
        \item $\mathbf{i}_{[-j,-i]}$ as the word obtained by the decreasing sequence from $j$ to \(i\), that is,
        \begin{equation*}
            \mathbf{i}_{[-j,-i]} = j\cdot(j-1)\cdots (i+1)\cdot i \in I^*.
        \end{equation*}
    \end{itemize}
    For integers \(a,b\) such that \(a>b\), we define \(\mathbf{i}_{[a,b]}\) to be the empty word.
\end{definition}
\begin{definition}
    For \(i,j\in I\), we define \(\mathbf{i}_{[-i,j]}\) to be the word
    obtained by concatenating the decreasing sequence from \(i\) to \(0\)
    with the increasing sequence from \(1\) to \(j\), that is,
    \begin{equation*}
        \mathbf{i}_{[-i,j]} = i\cdot(i-1)\cdots 1\cdot 0\cdot 1\cdots (j-1)\cdot j \in I^*.
    \end{equation*}
\end{definition}
For integers \(a,b,c\) satisfying
\begin{equation*}
    -(n-1)\leq a\leq b\leq c\leq n-1,
\end{equation*}
we have the identity
\begin{equation*}
    \mathbf{i}_{[a,c]}
    =
    \mathbf{i}_{[a,b]}\mathbf{i}_{[b+1,c]}.
\end{equation*}

\subsection{The symmetric group \texorpdfstring{$A_{n}$}{An}}
The rank-\(n\) Coxeter group \(W'=A_n\) has simple generators
\begin{equation*}
    S'=\{s'_i:0\leq i\leq n-1\}
\end{equation*}
and Coxeter matrix determined, for distinct \(i,j\in I\), by
\begin{equation*}
    m'(i,j)
    =
    \begin{cases}
        3,
            &\lvert i-j\rvert=1,\\
        2,
            &\lvert i-j\rvert\geq2.
    \end{cases}
\end{equation*}

We realize \(A_n\) as the group of permutations of
\begin{equation*}
    \{0,1,\ldots,n\},
\end{equation*}
and write \(w\in A_n\) in one-line notation as
\begin{equation*}
    w=w_0w_1\ldots w_n.
\end{equation*}
For \(0\leq i\leq n-1\), right multiplication by \(s'_i\)
interchanges the entries \(w_i\) and \(w_{i+1}\).

Before defining the canonical Lehmer code, we analyze the structure of canonical words.
For \(1\leq k\leq n\), we have
\begin{equation*}
    I_k=\{0,1,\ldots,k-1\}.
\end{equation*}
The corresponding standard parabolic subgroup \(W'_k\) permutes
\(\{0,1,\ldots,k\}\) and fixes \(\{k+1,\ldots,n\}\). In particular, \(W'_k\) is a Coxeter group of type \(A_k\).
For \(k=n\), we recover \(I_n=I\) and \(W'_n=W'\).

The set of minimal representatives of the right cosets \(W'_{k-1}\backslash W'_k\) is
\begin{equation*}
    {}^{\langle k-1\rangle}W'
    :=
    \{
        \mathrm{id},
        s'_{k-1},
        s'_{k-1}s'_{k-2},
        \ldots,
        s'_{k-1}s'_{k-2}\cdots s'_0
    \}.
\end{equation*}
Thus, \(\lvert{}^{\langle k-1 \rangle}W'\rvert=k+1\).
Each element of \({}^{\langle k-1\rangle}W'\) has a unique reduced word, and these reduced words are precisely the prefixes of \(\mathbf i_{[-(k-1),0]}\), with the identity corresponding to the empty prefix.

Let \(w'_0\) be the longest element of~\(A_n\). By the preceding
description of the minimal coset representatives, for every \(w\in A_n\),
the canonical word \(\mathbf c'(w)\) is a subword of
\begin{equation*}
    \mathbf c'(w'_0)
    =
    \mathbf i_{[0,0]}\cdot
    \mathbf i_{[-1,0]}\cdots
    \mathbf i_{[-(n-1),0]},
\end{equation*}
and has the following form for some \(q\geq 0\), with the product
understood to be empty when \(q=0\):
\begin{equation*}
    \mathbf{c}'(w)
    =
    \mathbf{i}_{[n_1,m_1]}
    \mathbf{i}_{[n_2,m_2]}
    \cdots
    \mathbf{i}_{[n_q,m_q]}.
\end{equation*}
If \(q\geq1\), the indices satisfy:
\begin{itemize}
    \item \(0\geq n_1>n_2>\cdots>n_q\geq-(n-1)\);
    \item \(0\geq m_j\geq n_j\) for every \(1\leq j\leq q\).
\end{itemize}
For each \(1\leq j\leq q\), the word \(\mathbf{i}_{[n_j,m_j]}\) is the reduced word of the canonical factor \(w^{(\lvert n_j\rvert+1)}\).
Consequently, for \(1\leq j\leq q\),
\begin{equation*}
    \ell\bigl(w^{(\lvert n_j\rvert+1)}\bigr)
    =
    \lvert n_j\rvert+1+m_j.
\end{equation*}

We recall the following result of Edelman~\cite[Theorem~2.3]{Ede95}.
\begin{proposition}\label{prop:canonical_is_lex_min}
    For every \(w\in A_n\), the canonical word \(\mathbf{c}'(w)\) is the
    lexicographically minimal element of \(\mathcal{R}(w)\).
\end{proposition}

With the left coset convention of~\cite[Section~1.3]{Ste97}, the canonical word of \(w\in A_n\) is the
lexicographically minimal element of \(\mathcal R(w)\) when words are
read from right to left.

\begin{definition}\label{def:lehmer code for An LAn}
    The \emph{canonical Lehmer code for~\(A_n\)} is the map
    \begin{gather*}
        L_{A_n}\colon A_n\longrightarrow
        \prod_{i=1}^{n}\{0,1,\ldots,i\},\\
        w\longmapsto
        \bigl(
            \ell(w^{(1)}),\ell(w^{(2)}),\ldots,\ell(w^{(n)})
        \bigr).
    \end{gather*}
\end{definition}
By~\cite[Section~5.2]{BS25}, \(L_{A_n}\) is a Lehmer code for \(A_n\).

Let $\mathrm{swap}$ be the involution of~$I^*$ sending \(i\) to $n-1-i$, that is, 
\begin{equation*}
    \mathrm{swap}(i_1i_2\cdots i_q)=j_1j_2\cdots j_q, \text{where $j_k=n-1-i_k$}.
\end{equation*}
This defines the automorphism $\sigma(w)= w'_0 w w'_0$ of~$A_n$.
Equivalently, $\sigma(w)=\mathbf{s}'(\mathrm{swap}(\mathbf{c}'(w)))$.

We end this subsection by defining a second distinguished reduced word \(\mathbf r'(w)\) for each \(w\in A_n\).
We use this word in the next section to relate the bottom map to
Bruhat order.

\begin{definition}\label{def:my_reduced_word_in_An}
    For \(w\in A_n\), define the word $\mathbf{r}'(w)$ by
    \begin{equation*}
        \mathbf{r}'(w)=\mathrm{swap}(\mathbf{c}'(\sigma(w))).
    \end{equation*}
\end{definition}

For example,
\[
    \mathbf{r}'(w'_0)= \mathbf{i}_{[n-1,n-1]} \cdot \mathbf{i}_{[n-2,n-1]} \cdots \mathbf{i}_{[0,n-1]}.
\]

The preceding description and
Proposition~\ref{prop:canonical_is_lex_min} give the following.
\begin{corollary}\label{cor:my_word_is_lex_max}
    For every \(w\in A_n\), the word \(\mathbf r'(w)\) is the
    lexicographically maximal element of \(\mathcal R(w)\). Moreover,
    \(\mathbf r'(w)\) is a subword of \(\mathbf r'(w'_0)\).
\end{corollary}

\subsection{The hyperoctahedral group \texorpdfstring{$B_n$}{Bn}}

The rank-\(n\) Coxeter group \(W=B_n\) has simple generators
\begin{equation*}
    S=\{s_i:0\leq i\leq n-1\}
\end{equation*}
and Coxeter matrix determined, for distinct \(i,j\in I\), by
\begin{equation*}
    m(i,j)
    =
    \begin{cases}
        4,
            &\{i,j\}=\{0,1\},\\
        3,
            &\lvert i-j\rvert=1
             \text{ and }\{i,j\}\neq\{0,1\},\\
        2,
            &\lvert i-j\rvert\geq2.
    \end{cases}
\end{equation*}

We realize \(B_n\) as the group of signed permutations and write \(w\in B_n\) in one-line notation as
\begin{equation*}
    w=w_1\cdots w_n,
\end{equation*}
where \(w_1,\ldots,w_n\in\{\pm1,\ldots,\pm n\}\) and
\begin{equation*}
    \{
        \lvert w_1\rvert,\lvert w_2\rvert,\ldots,\lvert w_n\rvert
    \}
    =
    \{1,2,\ldots,n\}.
\end{equation*}
For a signed entry \(a\), we write \(\overline a=-a\).

For \(1\leq i\leq n-1\), right multiplication by \(s_i\) interchanges the entries in positions \(i\) and \(i+1\), whereas \(s_0\) changes the sign of the first entry:
\begin{equation*}
	w_1w_2\cdots w_n \cdot s_0 = \overline{w_1}w_2\cdots w_n.
\end{equation*}

We now analyze the structure of canonical words.

For \(1\leq k\leq n\), the standard parabolic subgroup \(W_k\) associated with \(I_k\) permutes \(\{\pm1,\ldots,\pm k\}\) and fixes
\(\{\pm(k+1),\ldots,\pm n\}\).
In particular, \(W_k\) is a Coxeter group
of type \(B_k\), and \(W_n=W\).

For \(0\leq i\leq n-1\) and \(\lvert j\rvert\leq i\), the word
\(\mathbf{i}_{[-i,j]}\) is reduced in \(B_n\).
We define
\begin{equation*}
    \mathbf{s}_{[-i,j]}
    :=
    \mathbf{s}\bigl(\mathbf{i}_{[-i,j]}\bigr).
\end{equation*}

The set of minimal representatives of the right cosets
\(W_{k-1}\backslash W_k\) is
\begin{equation*}
    {}^{\langle k-1\rangle}W
    :=
    \{\mathrm{id}\}
    \cup
    \bigl\{
        \mathbf{s}_{[-(k-1),j]}:
        -(k-1)\leq j\leq k-1
    \bigr\}.
\end{equation*}
Thus, \(\lvert{}^{\langle k-1\rangle}W\rvert=2k\).
Each element of \({}^{\langle k-1\rangle}W\) has a unique reduced word,
and these reduced words are precisely the prefixes of
\(\mathbf{i}_{[-(k-1),k-1]}\), with the identity corresponding to the empty prefix.

Let \(w_0\) be the longest element of~\(B_n\). By the preceding
description of the minimal coset representatives, for every \(w\in B_n\),
the canonical word \(\mathbf c(w)\) is a subword of
\begin{equation*}
    \mathbf c(w_0)
    =
    \mathbf i_{[0,0]}\cdot
    \mathbf i_{[-1,1]}\cdots
    \mathbf i_{[-(n-1),n-1]},
\end{equation*}
and has the following form for some \(q\geq0\), with the product
understood to be empty when \(q=0\):
\begin{equation}\label{eq:canonical_word_type_B}
    \mathbf c(w)
    =
    \mathbf i_{[n_1,m_1]}\cdots\mathbf i_{[n_q,m_q]}.
\end{equation}
If \(q\geq1\), the indices satisfy:
\begin{itemize}
    \item \(0\geq n_1>\cdots>n_q\geq-(n-1)\);
    \item \(\lvert m_j\rvert\leq\lvert n_j\rvert\) for every
    \(1\leq j\leq q\).
\end{itemize}
For each \(1\leq j\leq q\), the word
\(\mathbf i_{[n_j,m_j]}\) is the reduced word of the canonical factor
\(w^{(\lvert n_j\rvert+1)}\). Since this word has
\(m_j-n_j+1=\lvert n_j\rvert+1+m_j\) letters, we obtain, as in \(A_n\),
\begin{equation*}
    \ell\bigl(w^{(\lvert n_j\rvert+1)}\bigr)
    =
    \lvert n_j\rvert+1+m_j.
\end{equation*}

\begin{definition}\label{def:lehmer code for Bn LBn}
    The \emph{canonical Lehmer code for~\(B_n\)} is the map
    \begin{gather*}
        L_{B_n}\colon B_n\longrightarrow
        \prod_{i=1}^{n}\{0,1,\ldots,2i-1\},\\
        w\longmapsto
        \bigl(
            \ell(w^{(1)}),\ell(w^{(2)}),\ldots,\ell(w^{(n)})
        \bigr).
    \end{gather*}
\end{definition}
By~\cite[Theorem~5.4]{BS25}, \(L_{B_n}\) is a Lehmer code for \(B_n\).

In the next lemma, we show how to recover the one-line notation of a signed permutation \(w\) directly from its code \(L_{B_n}(w)\).
More precisely, the \(k\)-th entry of the code determines how the new largest absolute value \(k\) is inserted into the one-line notation constructed from the first \(k-1\) entries.
This insertion procedure is classical; see, for example, the final part of the proof of~\cite[Chapter~7, Section~2.2, Proposition~2.11]{Win00}.

\begin{lemma}\label{lem:LBn_insertion_description}
    Let \(w\in B_n\) and \(L_{B_n}(w)=(a_1,\ldots,a_n)\).
    Then the one-line notation of~\(w\) is obtained recursively by inserting the new largest absolute value \(k\) at step \(k\) as follows:
    \begin{equation*}
        \begin{array}{ccl}
            0\leq a_k\leq k-1
            &\Longrightarrow&
            \text{insert }k\text{ in position }k-a_k,\\
            k\leq a_k\leq2k-1
            &\Longrightarrow&
            \text{insert }\sbar k\text{ in position }a_k-k+1.
        \end{array}
    \end{equation*}
    Here, at step \(k\), positions are counted from left to right among the \(k\) positions of the resulting word.
\end{lemma}
\begin{proof}
    At step \(k\), append \(k\) to the word obtained after step \(k-1\)
    and multiply on the right by the \(k\)-th canonical factor
    \(w^{(k)}\).
    If \(a_k=0\), this factor is the identity, so \(k\) remains
    in position \(k\).

    Suppose that \(a_k>0\), and set \(m=a_k-k\).
    The reduced word of \(w^{(k)}\) is
    \(\mathbf{i}_{[-(k-1),m]}\).

    If \(1\leq a_k\leq k-1\), then \(m<0\), and
    \begin{equation*}
        w^{(k)}=s_{k-1}s_{k-2}\cdots s_{-m}.
    \end{equation*}
    Thus \(k\) moves left from position \(k\) to position
    \(-m=k-a_k\) and remains positive.

    If \(k\leq a_k\leq2k-1\), then \(m\geq0\), and
    \begin{equation*}
        w^{(k)}
        =
        (s_{k-1}s_{k-2}\cdots s_1)\,
        s_0\,
        (s_1s_2\cdots s_m),
    \end{equation*}
    where either parenthesized product may be empty.
    The first product moves \(k\) to the first position,
    \(s_0\) changes it to \(\overline{k}\), and the last product
    moves it right to position \(m+1=a_k-k+1\).

    In both cases, the other entries retain their signs and relative
    order, giving the stated insertions.
\end{proof}

\begin{definition}\label{def:ell_0}
    For \(w=w_1\cdots w_n\in B_n\), we denote the number of negative entries of \(w\) by
    \begin{equation*}
        \ell_0(w)
        :=
        \bigl|\{i\in\{1,\ldots,n\}:w_i<0\}\bigr|.
    \end{equation*}
\end{definition}

By Lemma~\ref{lem:LBn_insertion_description}, if
\(L_{B_n}(w)=(a_1,\ldots,a_n)\), then
\begin{equation*}
    \ell_0(w)=\bigl|\{k:1\leq k\leq n,\ a_k\geq k\}\bigr|.
\end{equation*}
This is also the number of occurrences of \(0\) in the canonical word \(\mathbf c(w)\), since its \(k\)-th factor contains exactly one \(0\) when \(a_k\geq k\), and none otherwise.
Since any two reduced words are related by braid moves, and braid moves in \(B_n\) preserve the number of occurrences of \(0\), every reduced word of \(w\) contains exactly \(\ell_0(w)\) occurrences of \(0\).

\begin{example}
    Let \(w\in B_7\) satisfy
    \begin{equation*}
        L_{B_7}(w)=(1,0,2,6,1,3,13).
    \end{equation*}
    Its canonical word is
    \begin{equation*}
        \mathbf c(w)
        =
        0 \cdot 21 \cdot 321012 \cdot 4 \cdot 543 \cdot 6543210123456,
    \end{equation*}
    where the dots separate the reduced words of the non-trivial canonical factors.
    Lemma~\ref{lem:LBn_insertion_description} gives the successive insertions
    \begin{equation*}
        \overline{1} \rightarrow \overline{1}2 \rightarrow 3\overline{1}2 \rightarrow 3\overline{1}\overline{4}2 \rightarrow 3\overline{1}\overline{4}52 \rightarrow 3\overline{1}6\overline{4}52 \rightarrow 3\overline{1}6\overline{4}52\overline{7}. 
    \end{equation*}
    Thus, \(w=3\overline1 6\overline4 52\overline7\) and
    \(\ell_0(w)=3\).
\end{example}

\subsection{Comparing reduced words in~\texorpdfstring{\(A_n\)}{An} and~\texorpdfstring{\(B_n\)}{Bn}}

Since \(B_n\) dominates \(A_n\), the first sentence of~\cite[Proposition~1.2]{Ste97} gives
\[
    \mathcal{R}(A_n)\subseteq \mathcal{R}(B_n).
\]
The following proposition combines~\cite[Proposition~4.2]{Ste97} with the discussion preceding it.
\begin{proposition}\label{prop: rex of x inter min zeros}
    For each \(x\in A_n\), there is a unique \(w\in B_n\) such that
    \begin{equation}\label{eq: rex of x inter min zeros}
        \mathcal{R}(w)
        =
        \{\mathbf{i}\in\mathcal{R}(x):
        \text{the number of zeros in \(\mathbf{i}\) is minimal}\}.
    \end{equation}
    Conversely, if \(w\in B_n\) satisfies \(\ell_0(w)\leq 1\), then there
    is a unique \(x\in A_n\) such that
    \eqref{eq: rex of x inter min zeros} holds.
\end{proposition}

\begin{definition}\label{def:bottom map and bottom element}
    For \(x\in A_n\), the \emph{bottom element associated with \(x\)}, denoted by \(\operatorname{bot}(x)\), is the unique element of \(B_n\) supplied by Proposition~\ref{prop: rex of x inter min zeros}.
    The resulting map
    \begin{equation*}
        \operatorname{bot}\colon A_n\longrightarrow B_n
    \end{equation*}
    is called the \emph{bottom map}.
\end{definition}

This definition agrees with the one-line description of the bottom map given in the introduction.

\begin{example}
    For \(n=7\), consider \(x=72106534\in A_7\). By the one-line description of the bottom map,
    \[
        \operatorname{bot}(x)=21\overline{7}6534\in B_7.
    \]
    A SageMath computation shows that \(x\) has
    \(1{,}117{,}025\) reduced words, of which \(67{,}795\) contain exactly one zero.
    These are precisely the reduced words of \(\operatorname{bot}(x)\).
\end{example}
\section{The Bruhat poset of \texorpdfstring{\(A_n\)}{An} as a lower interval in \texorpdfstring{\(B_n\)}{Bn}}\label{sec:bottom-map}

Finite standard parabolic subgroups provide a simple source of lower Bruhat intervals.
Indeed, if \(W_P\subseteq W\) is a finite standard parabolic subgroup with longest element \(w_P\), then \begin{equation*}
    (W_P,\leq)=[\mathrm{id},w_P]
\end{equation*}
as subposets of \(W\).

We now prove Theorem~\ref{thm: iso posets An to Bn intro},
which gives such a realization beyond the standard parabolic setting.
We first relate the bottom map to the reduced words \(\mathbf r'(x)\) of Definition~\ref{def:my_reduced_word_in_An}.

\begin{lemma}\label{lem:reduced words of bottom elements}
    For every \(x\in A_n\),
    \begin{equation*}
        \mathcal R(\operatorname{bot}(x))
        =
        \{
            \mathbf i\in\mathcal R(x):
            \mathbf i\text{ contains at most one zero}
        \}.
    \end{equation*}
    In particular,
    \begin{equation*}
        \mathbf r'(x)\in\mathcal R(\operatorname{bot}(x)),
        \qquad
        \operatorname{bot}(x)
        =
        \mathbf s\bigl(\mathbf r'(x)\bigr).
    \end{equation*}
\end{lemma}
\begin{proof}
    By Corollary~\ref{cor:my_word_is_lex_max},
    \(\mathbf r'(x)\) is a reduced word of \(x\) and a subword of
    \(\mathbf r'(w'_0)\), which contains exactly one zero.
    Hence the minimum number of zeros among the reduced words of
    \(x\) is at most one.

    Any two reduced words of \(x\) are related by braid moves,
    which preserve the set of indices occurring.
    If one reduced word contains no zero, then none does, and
    every reduced word attains the minimum.
    Otherwise, every reduced word contains at least one zero,
    so the minimum is one and is attained precisely by those
    containing exactly one zero.

    In either case, the reduced words attaining the minimum are
    precisely those containing at most one zero.
    Proposition~\ref{prop: rex of x inter min zeros} therefore gives
    the stated equality. Since \(\mathbf r'(x)\) contains at most
    one zero, the remaining assertions follow.
\end{proof}

\begin{proof}[Proof of Theorem~\ref{thm: iso posets An to Bn intro}]
    The one-line description of the bottom map gives
    \begin{equation*}
        \operatorname{bot}(w'_0)
        =
        \operatorname{bot}\bigl(n(n-1)\cdots 10\bigr)
        =
        (n-1)(n-2)\cdots1\,\sbar n
        =
        b_0.
    \end{equation*}
    
    We first show that \(\operatorname{bot}\) preserves and reflects
    Bruhat order.
    Let \(x,y\in A_n\).
    By Lemma~\ref{lem:reduced words of bottom elements},
    \(\mathbf r'(y)\) is a reduced word of both \(y\) and
    \(\operatorname{bot}(y)\).
    Every subword of \(\mathbf r'(y)\) contains at most one zero,
    so the same lemma implies that such a subword belongs to
    \(\mathcal R(x)\) if and only if it belongs to
    \(\mathcal R(\operatorname{bot}(x))\).
    Applying the Subword Property in both groups gives
    \begin{equation*}
        x\leq y
        \quad\Longleftrightarrow\quad
        \operatorname{bot}(x)\leq\operatorname{bot}(y).
    \end{equation*}
    Thus, \(\operatorname{bot}\) is a poset embedding.

    Lemma~\ref{lem:reduced words of bottom elements} and the converse
    part of Proposition~\ref{prop: rex of x inter min zeros} give
    \begin{equation*}
        \operatorname{bot}(A_n)
        =
        \{w\in B_n:\ell_0(w)\leq1\}.
    \end{equation*}
    Since \(x\leq w'_0\) for every \(x\in A_n\), order preservation
    gives \(\operatorname{bot}(A_n)\subseteq[\mathrm{id},b_0]\).
    Conversely, let \(w\leq b_0\). By the Subword Property, some
    reduced word of \(w\) is a subword of
    \(\mathbf r'(w'_0)\in\mathcal R(b_0)\).
    This word contains at most one zero, so \(\ell_0(w)\leq1\).
    Hence \(w\in\operatorname{bot}(A_n)\), proving that
    \(\operatorname{bot}(A_n)=[\mathrm{id},b_0]\).
\end{proof}

\begin{remark}[Other finite and affine types]
\label{rem:other-types-search}
    Using SageMath, we checked whether a poset embedding with lower-interval image, as in Theorem~\ref{thm: iso posets An to Bn intro}, exists in several other Coxeter types.
    More precisely, for a Coxeter group \(W\) of rank \(n\), we searched
    for elements \(w\in W\) such that \([\mathrm{id},w]\) is isomorphic
    to the Bruhat poset \((A_n,\leq)\).
    Apart from the rank-two dihedral family, we found only isolated
    low-rank examples among the cases examined; no further uniform
    family emerged.
    \begin{itemize}
        \item \textbf{Crystallographic finite types.}
        Among the remaining irreducible crystallographic finite types
        of ranks \(4\), \(5\), and \(6\), no such intervals occur.
        Namely, in types \(D_4\), \(D_5\), and \(D_6\), there are no
        lower intervals isomorphic to \((A_4,\leq)\),
        \((A_5,\leq)\), and \((A_6,\leq)\), respectively.
        Similarly, in types \(F_4\) and \(E_6\), there are no lower
        intervals isomorphic to \((A_4,\leq)\) and \((A_6,\leq)\),
        respectively.

        \item \textbf{Types \(H_3\) and \(H_4\).}
        In these non-crystallographic finite types, such intervals do
        occur. For \(H_n\), where \(n\in\{3,4\}\), we use the index set
        \(\{1,\ldots,n\}\), with \(m(1,2)=5\).
        Each type contains a unique such element; reduced words for
        these elements are \(123121\) and \(1234123121\), respectively.

        \item \textbf{Type \(I_2(m)\).}
        With the simple generators indexed by \(1\) and \(2\), for every
        \(m\geq3\), the element with reduced word \(121\) has a lower
        interval isomorphic to \((A_2,\leq)\).

        \item \textbf{Affine types.}
        We use the standard affine indexing, with the affine simple
        generator indexed by \(0\). Since an isomorphism of lower
        intervals preserves rank, it suffices to check elements of
        length \(6\) in the types with subscript \(2\) and elements of
        length \(10\) in those with subscript \(3\).

        No such interval occurs in \(\widetilde A_2\) or
        \(\widetilde B_2\). In \(\widetilde G_2\), however, the element
        with reduced word \(020120\) gives such an interval.
        No such interval occurs in \(\widetilde A_3\),
        \(\widetilde B_3\), or \(\widetilde C_3\).
    \end{itemize}
\end{remark}
\section{Reflected bounded-growth tuples}\label{sec:rbg}

We now study in more detail the class of tuples indexing the type~\(B\)
family in our classification. For convenience, we recall its definition.

\begin{definition}\label{def: rbg and rbg tuples}
    An \(n\)-tuple \(\mathbf{a}=(a_1,a_2,\ldots,a_n)\in\mathbb Z^n\) has
    \emph{reflected bounded-growth}, or is an \emph{rbg tuple}, if
    \begin{itemize}
        \item \(0\leq a_i\leq 2i-1\) for all \(i\);
        \item for every \(2\leq i\leq n\), if \(a_{i-1}<2i-3\), then
        \[
            a_i\leq
            \min\{a_{i-1}+1,\,2i-3-a_{i-1}\}.
        \]
    \end{itemize}
    The second condition is called the \emph{reflected bound}.
    We denote by \(\mathcal C_n\) the set of all rbg \(n\)-tuples.
\end{definition}

The terminology refers to the second term in the minimum:
\(2i-3-a_{i-1}\) is obtained by reflecting \(a_{i-1}\) through
\(i-\frac{3}{2}\).

Following Bolognini and Sentinelli~\cite[Definition~7.6]{BS25}, we make the following definition.
\begin{definition}\label{def:lazy_fubini}
    The set of \emph{lazy Fubini words} of length \(n\) is
    \begin{equation*}
        \widehat{F}_n
        :=
        \left\{
            (a_1,\ldots,a_n):
            a_1=0
            \text{ and }
            0\leq a_{i+1}\leq a_i+1
            \text{ for }1\leq i<n
        \right\}.
    \end{equation*}
\end{definition}
Lazy Fubini words are enumerated by the Catalan numbers~\cite[Exercise~6.19(u)]{Sta99}:
\begin{equation*}
    \lvert\widehat{F}_n\rvert
    =
    C_n
    =
    \frac{1}{n+1}\binom{2n}{n}.
\end{equation*}

When the first entry is \(0\), the reflected bound becomes redundant,
giving the following identification.
\begin{lemma}\label{lem: rbg starting at 0 are all lazy fubini words}
    We have
    \begin{equation}\label{eq: rbg int 0 equals lazy fubini}
        \mathcal C_n\cap\{\mathbf a:a_1=0\}
        =
        \widehat F_n.
    \end{equation}
    In particular, every lazy Fubini word is an rbg tuple.
\end{lemma}

\begin{proof}
    Let \(\mathbf a=(a_1,\ldots,a_n)\in\mathcal C_n\) satisfy \(a_1=0\).
    An induction using the rbg condition gives
    \begin{equation*}
        a_i\leq i-1\quad(1\leq i\leq n),
    \end{equation*}
    In particular, for \(2\leq i\leq n\), the rbg condition applies and
    gives \(a_i\leq a_{i-1}+1\). Hence
    \(\mathbf a\in\widehat F_n\).

    Conversely, let \(\mathbf a\in\widehat F_n\). Again,
    \(a_i\leq i-1\) for every \(i\). Therefore
    \begin{equation*}
        a_i\leq a_{i-1}+1
        \qquad\text{and}\qquad
        a_i\leq i-1\leq 2i-3-a_{i-1}
    \end{equation*}
    for \(2\leq i\leq n\). Thus \(\mathbf a\) satisfies the rbg
    condition.
\end{proof}

We count rbg tuples.
\begin{proposition}\label{prop: counting rbg}
    \(\lvert\mathcal C_n\rvert=C_n+C_{n+1}-1\).
\end{proposition}
\begin{proof}
    By Lemma~\ref{lem: rbg starting at 0 are all lazy fubini words},
    the number of tuples in \(\mathcal C_n\) with \(a_1=0\) is \(C_n\).
    It remains to count those with \(a_1=1\).

    For \(0\leq m\leq2n-1\), let
    \begin{equation*}
        F_n(m)
        :=
        \bigl|
            \{
                (a_1,\ldots,a_n)\in\mathcal C_n:
                a_1=1,\ a_n=m
            \}
        \bigr|.
    \end{equation*}
    We have \(F_1(0)=0\) and \(F_1(1)=1\). For \(n\geq2\), conditioning
    on \(a_{n-1}=x\) gives
    \begin{equation}\label{eq:Fn-recurrence}
        F_n(m)
        =
        1+
        \sum_{x=\max\{0,m-1\}}^{\min\{2n-4,\,2n-3-m\}}
        F_{n-1}(x),
    \end{equation}
    where an empty sum is zero. Indeed, the initial \(1\) corresponds to
    the unique prefix
    \begin{equation*}
        (1,3,\ldots,2n-3),
    \end{equation*}
    while, for \(x<2n-3\), the rbg condition is equivalent to
    \begin{equation*}
        m-1\leq x\leq2n-3-m.
    \end{equation*}
    In particular,
    \begin{equation*}
        F_n(m)=1
        \qquad(n\leq m\leq2n-1),
        \qquad
        F_n(0)=F_n(1).
    \end{equation*}

    Define
    \begin{equation*}
        G_n(m)
        :=
        \begin{cases}
            1+F_n(m),&0\leq m\leq n-1,\\
            1,&m=n,\\
            0,&m<0.
        \end{cases}
    \end{equation*}
    Equation~\eqref{eq:Fn-recurrence} gives
    \begin{equation}\label{eq:Gn-catalan-recurrence}
        G_n(m)
        =
        \sum_{x=m-1}^{n-1}G_{n-1}(x)
        \qquad(0\leq m\leq n).
    \end{equation}
    For \(1\leq m\leq n-2\), this follows by separating the sum
    in~\eqref{eq:Fn-recurrence} at \(x=n-2\):
    \begin{align*}
        G_n(m)
        &=
        2+\sum_{x=m-1}^{2n-3-m}F_{n-1}(x)\\
        &=
        1+\sum_{x=m-1}^{n-2}G_{n-1}(x)
        =
        \sum_{x=m-1}^{n-1}G_{n-1}(x).
    \end{align*}
    The cases \(m=0,n-1,n\) follow directly from
    \(F_n(0)=F_n(1)\), Equation~\eqref{eq:Fn-recurrence}, and the
    definition of \(G_n\), respectively.

    Together with \(G_1(0)=G_1(1)=1\),
    \eqref{eq:Gn-catalan-recurrence} is the Catalan triangle recurrence.
    Hence
    \begin{equation*}
        \sum_{m=0}^nG_n(m)=C_{n+1}.
    \end{equation*}
    Since \(G_n(n)=1\), the number of rbg tuples with \(a_1=1\) is
    \begin{align*}
        \sum_{m=0}^{2n-1}F_n(m)
        &=
        n+\sum_{m=0}^{n-1}F_n(m)\\
        &=
        \sum_{m=0}^{n-1}G_n(m)
        =
        C_{n+1}-1.
    \end{align*}
    Adding the \(C_n\) tuples with \(a_1=0\) proves the result.
\end{proof}
\section{\texorpdfstring{$L_{B_n}$}{LBn}-principal elements}

In this section, we prove that the Kempf elements in~\(B_n\), equivalently
the \(L_{B_n}\)-principal elements, are exactly those whose \(L_{B_n}\)-codes
are rbg tuples.
We then give an equivalent characterization in terms of pattern avoidance.

We first compare the type~\(A\) and type~\(B\) constructions. Consider the
standard parabolic subgroup
\begin{equation*}
    S_n=\langle s_1,\ldots,s_{n-1}\rangle\subset B_n.
\end{equation*}
For every \(w\in S_n\), the canonical Lehmer codes satisfy
\begin{equation*}
    L_{B_n}(w)=\bigl(0,L_{S_n}(w)\bigr).
\end{equation*}
Moreover, if \(u\leq w\) in \(B_n\), then \(u\in S_n\). Indeed, \(w\) has a
reduced expression using only the generators \(s_1,\ldots,s_{n-1}\), so the
claim follows from the Subword Property (Lemma~\ref{lem:Subword Property}).
Consequently,
\begin{equation*}
    \mathrm{Pr}(L_{S_n})
    =
    \mathrm{Pr}(L_{B_n})\cap S_n.
\end{equation*}

By~\cite[Proposition~7.8]{BS25}, for every \(w\in S_n\), the following
conditions are equivalent:
\begin{equation}\label{eq:type-A-principal-code-pattern}
    w\in\mathrm{Pr}(L_{S_n})
    \quad\Longleftrightarrow\quad
    \bigl(0,L_{S_n}(w)\bigr)\in\widehat F_n
    \quad\Longleftrightarrow\quad
    \text{\(w\) avoids \(312\)}.
\end{equation}
Moreover, Lemma~\ref{lem: rbg starting at 0 are all lazy fubini words}
gives
\begin{equation*}
    \widehat F_n
    =
    \mathcal C_n\cap\{a_1=0\}.
\end{equation*}

We use pattern avoidance for signed permutations as in~\cite[Definition~4.1]{Bil98}. Consider the following eight signed
patterns in \(B_2\) and \(B_3\):
\begin{equation}\label{eq: 8 bad patterns}
    \sbar{2}1,\quad
    1\sbar{2},\quad
    \sbar{3}\sbar{2}\sbar{1},\quad
    \sbar{2}\sbar{3}\sbar{1},\quad
    \sbar{2}\sbar{1}\sbar{3},\quad
    3\sbar{2}\sbar{1},\quad
    3\sbar{1}2,\quad
    312.
\end{equation}

The theorem below may be viewed as a type~\(B\) extension of all three equivalences above.
Indeed, under the standard inclusion \(S_n\subset B_n\), the preceding identities show that its first two conditions restrict to the first two conditions in~\eqref{eq:type-A-principal-code-pattern}.
Moreover, since the elements of \(S_n\) are unsigned, they automatically avoid the seven signed patterns in~\eqref{eq: 8 bad patterns}; hence the third condition reduces to avoidance of \(312\).

\begin{theorem}\label{thm:equivalence of rbg codes}
    Let \(L_{B_n}\) be as in Definition~\ref{def:lehmer code for Bn LBn}.
    For \(w\in B_n\), the
    following statements are equivalent:
    \begin{enumerate}[label=(\roman*)]
        \item \(w\) is \(L_{B_n}\)-principal.
        \label{thm:equivalence of rbg codes-i}
        \item \(L_{B_n}(w)\) is an rbg tuple.
        \label{thm:equivalence of rbg codes-ii}
        \item \(w\) avoids the eight patterns in~\eqref{eq: 8 bad patterns}.
        \label{thm:equivalence of rbg codes-iii}
    \end{enumerate}
    In particular, \( \left\lvert\mathrm{Pr}(L_{B_n})\right\rvert =C_n+C_{n+1}-1\).
\end{theorem}

We prove the two equivalences in
Theorem~\ref{thm:equivalence of rbg codes} separately. The first
identifies the type~\(B\) family with the \(L_{B_n}\)-principal elements
and follows by translating Lakshmibai's definition of Kempf elements.

\begin{proposition}\label{prop:kempf-rbg}
    Let \(w\in B_n\). Then \(w\) is \(L_{B_n}\)-principal if and only if
    \(L_{B_n}(w)\) is an rbg tuple.
\end{proposition}
\begin{proof}
The case \(n=1\) is immediate, so assume \(n\geq2\).
    By~\cite[Example~6.4]{BS25}, the \(L_{B_n}\)-principal elements are
    precisely the Kempf elements in~\(B_n\). We translate Lakshmibai's
    canonical factorization and condition~\((K)\) \cite[Proposition~B.1 and condition~\((K)\) on p.~324]{Lak76}.
    
    Write \(\sigma_1,\ldots,\sigma_n\) for the simple reflections in
    Lakshmibai's numbering, so that
    \[
        \sigma_i=s_{n-i}\qquad (1\leq i\leq n)
    \]
    in our numbering. Lakshmibai writes
    \[
        w=u_n'u_{n-1}'\cdots u_1',
    \]
    where \(u_i'\), for \(1\leq i\leq n\), is one of
    \[
        \mathrm{id},\quad
        \sigma_i\sigma_{i+1}\cdots\sigma_j
        \,(i\leq j\leq n-1),\quad
        \sigma_i\sigma_{i+1}\cdots\sigma_{n-1}\sigma_n\sigma_{n-1}\cdots\sigma_{j+1}\sigma_j
        \,(i\leq j\leq n).
    \]
    Here and below, the descending string after \(\sigma_n\) is understood to be empty when \(j=n\).
    These are precisely the prefixes of the reduced expression
    \[
        u_i=\sigma_i\sigma_{i+1}\cdots\sigma_{n-1}\sigma_n\sigma_{n-1}\cdots\sigma_{i+1}\sigma_i.
    \]
    Hence they form a chain in Bruhat order, with one element of each
    length from \(0\) to \(2(n-i)+1\).
    
    Under our identification, the factors
    \(u_n',\ldots,u_1'\) correspond, in the same order, to the factors
    \(w^{(1)},\ldots,w^{(n)}\) in the canonical decomposition of \(w\).
    Thus, if
    \[
        L_{B_n}(w)=(a_1,\ldots,a_n)
    \]
    and \(k=n-i+1\), then
    \[
        a_k=\ell(u_i').
    \]
    
    By definition, \(w\) is a Kempf element if and only if it satisfies condition~\((K)\):
    \begin{enumerate}
        \item \(u'_n\) is arbitrary, that is, \(u'_n\in\{\mathrm{id},\sigma_n\}\).
        \item For \(1\leq i\leq n-1\):
        \begin{enumerate}
            \item \(u'_i\) is arbitrary if \(u'_{i+1}=u_{i+1}\).
            \item \(u'_i
            \leq \sigma_i\sigma_{i+1}\cdots\sigma_j\) if \(u'_{i+1}=\sigma_{i+1}\sigma_{i+2}\cdots\sigma_j \quad(i+1\leq j\leq n-1)\).
            \item \(u'_i
            \leq \sigma_i\sigma_{i+1}\cdots\sigma_{j-2}\) if
            \begin{equation*}
                u'_{i+1}=\sigma_{i+1}\sigma_{i+2}\cdots\sigma_{n-1}\sigma_n\sigma_{n-1}\cdots\sigma_{j+1}\sigma_j \quad(i+2\leq j\leq n).
            \end{equation*}
            \item \(u'_i\leq \sigma_i\) if \(u'_{i+1}=\mathrm{id}\).
        \end{enumerate}
    \end{enumerate}
    Since Bruhat order on each of these chains is determined by length,
    we can translate condition~\((K)\) directly. For \(2\leq k\leq n\),
    put \(b=a_{k-1}\). If \(b<2k-3\), condition~\((K)\) becomes
    \[
        a_k\leq
        \begin{cases}
            b+1, & 0\leq b\leq k-2,\\
            2k-3-b, & k-1\leq b\leq2k-4.
        \end{cases}
    \]
    Since
    \[
        b+1\leq2k-3-b
        \quad\Longleftrightarrow\quad
        b\leq k-2,
    \]
    this is equivalent to
    \[
        a_{k-1}<2k-3
        \quad\Longrightarrow\quad
        a_k\leq
        \min\{a_{k-1}+1,\,2k-3-a_{k-1}\}.
    \]
    For \(2\leq k\leq n\), if \(a_{k-1}=2k-3\), then
    condition~\((K)\) imposes no restriction on \(a_k\). For \(k=1\),
    condition~\((K)\) is automatic.
    
    Together with the \(L_{B_n}\)-code bounds
    \[
        0\leq a_k\leq2k-1\qquad(1\leq k\leq n),
    \]
    this shows that condition~\((K)\) is exactly the rbg condition.
\end{proof}

It remains to characterize rbg codes by pattern avoidance.
The following lemma reformulates the rbg inequalities using the insertion procedure from Lemma~\ref{lem:LBn_insertion_description}.

\begin{lemma}\label{lem:rbg-insertion-characterization}
    Let \(w\in B_n\) and let \(p_k\) be the position in which \(k\) or \(\sbar k\) is
    inserted at step~\(k\).
    Set \(r:=\ell_0(w)\) and, if \(r>0\), set \(t:=p_r\).
    Then \(L_{B_n}(w)\) is an rbg
    tuple if and only if the following conditions hold:
    \begin{enumerate}[
        label=\textup{(\roman*)},
        ref=\textup{(\roman*)},
        leftmargin=*
        ]
        \item\label{item:rbg-negative-subword}
        If \(r>0\), the absolute values of the negative entries of \(w\)
        are \(\{1,\ldots,r\}\), and their subword is obtained from
        \begin{equation*}
            \sbar1\,\sbar2\,\cdots\,\sbar r
        \end{equation*}
        by moving \(\sbar r\) to position \(t\).

        \item\label{item:rbg-positive-positions}
        For every \(r<k<n\),
        \begin{equation*}
            p_k\leq p_{k+1}.
        \end{equation*}

        \item\label{item:rbg-first-positive}
        If \(0<r<n\) and \(t<r\), then
        \begin{equation*}
            p_{r+1}\geq t+1.
        \end{equation*}
    \end{enumerate}
\end{lemma}

\begin{proof}
    Suppose first that \(L_{B_n}(w)=(a_1,\ldots,a_n)\) is an rbg tuple. By
    Lemma~\ref{lem:LBn_insertion_description}, \(a_k\geq k\) precisely
    when \(k\) is inserted negatively. If \(a_k\geq k\) for \(k\geq2\), then \(a_{k-1}=2k-3\).
    Indeed, if \(a_{k-1}<2k-3\), the rbg condition would give
    \(a_k\leq k-1\).
    Thus the set of absolute values of the negative entries is an initial
    segment of \(\{1,\ldots,n\}\). Since it has cardinality \(r\), it is
    \(\{1,\ldots,r\}\).

    The preceding argument also gives
    \begin{equation*}
        a_i=2i-1
        \qquad
        (1\leq i<r).
    \end{equation*}
    By Lemma~\ref{lem:LBn_insertion_description}, \(\sbar1,\ldots,\sbar{r-1}\) are inserted successively at the end.
    Since \(t=p_r\), the entry \(\sbar r\) is inserted in position
    \(t\), proving condition~\ref{item:rbg-negative-subword}.

    For \(k>r\), \(k\) is inserted positively.
    By Lemma~\ref{lem:LBn_insertion_description},
    \begin{equation*}
        a_k=k-p_k.
    \end{equation*}
    For \(r<k<n\), since \(a_k<k\), the rbg inequality at \(k+1\)
    becomes \(a_{k+1}\leq a_k+1\), which is equivalent to
    \(p_{k+1}\geq p_k\). This proves
    condition~\ref{item:rbg-positive-positions}.

    Finally, suppose that \(0<r<n\) and \(t<r\). Since
    \(a_r=r+t-1<2r-1\), the rbg inequality at \(r+1\) gives
    \begin{equation*}
        a_{r+1}\leq2r-1-a_r=r-t.
    \end{equation*}
    Since \(a_{r+1}=r+1-p_{r+1}\), this is equivalent to
    \(p_{r+1}\geq t+1\), proving
    condition~\ref{item:rbg-first-positive}.

    Conversely, suppose that
    conditions~\ref{item:rbg-negative-subword}--%
    \ref{item:rbg-first-positive} hold.
    
    If \(r>0\), condition~\ref{item:rbg-negative-subword} shows that,
    for \(i<r\), the entry \(\sbar i\) is inserted at the end, whereas
    \(\sbar r\) is inserted in position \(t\). Therefore,
    \begin{equation*}
        a_i=2i-1\quad(1\leq i<r),
        \qquad
        a_r=r+t-1.
    \end{equation*}
    By the definition of \(r\) and
    condition~\ref{item:rbg-negative-subword}, every entry \(k>r\) is
    positive, so
    \begin{equation*}
        a_k=k-p_k
        \qquad
        (k>r).
    \end{equation*}

    The maximal values \(a_i=2i-1\) impose no rbg restriction on the
    following coordinate. If \(0<r<n\) and \(t=r\), neither does
    \(a_r=2r-1\). If \(0<r<n\) and \(t<r\), then condition~\ref{item:rbg-first-positive} gives
    \begin{equation*}
        a_{r+1}
        \leq r-t
        =2r-1-a_r
        \leq a_r+1,
    \end{equation*}
    which is the rbg inequality at \(r+1\).

    For \(r<k<n\), condition~\ref{item:rbg-positive-positions} gives
    \(a_{k+1}\leq a_k+1\). Since \(a_k<k\),
    \begin{equation*}
        a_k+1\leq2k-1-a_k,
    \end{equation*}
    so this is precisely the rbg inequality at \(k+1\).

    The bounds \(0\leq a_i\leq 2i-1\) are automatic for an
    \(L_{B_n}\)-code.
    Therefore, \((a_1,\ldots,a_n)\) is an rbg tuple.
\end{proof}

\begin{example}
    Let \(w\in B_6\) satisfy
    \begin{equation*}
        L_{B_6}(w)=(1,3,3,2,1,2)\in\mathcal C_6.
    \end{equation*}
    The corresponding insertion procedure is
    \begin{equation*}
            \sbar1
            \longrightarrow \sbar1\sbar2
            \longrightarrow \sbar3\sbar1\sbar2
            \longrightarrow \sbar3 4\sbar1\sbar2
            \longrightarrow \sbar3 4\sbar1 5\sbar2
            \longrightarrow \sbar3 4\sbar1 6 5\sbar2.
    \end{equation*}
    Thus \(w=\sbar3 4\sbar1 6 5\sbar2\).

    In the notation of
    Lemma~\ref{lem:rbg-insertion-characterization}, we have
    \(r=3\), \(t=1\), and
    \begin{equation*}
        (p_1,p_2,p_3,p_4,p_5,p_6)=(1,2,1,2,4,4).
    \end{equation*}
    Moreover, \(p_4=2=t+1\), so condition~\ref{item:rbg-first-positive} holds with equality.
\end{example}

We will use the following elementary insertion criterion.

\begin{lemma}\label{lem:312-insertion-positions}
    Let \(v\in S_m\) be obtained by successively inserting the new maximum \(k\)
    in position \(q_k\). Then \(v\) avoids \(312\) if and only if
    \begin{equation*}
        q_1\leq q_2\leq\cdots\leq q_m.
    \end{equation*}
\end{lemma}

\begin{proof}
    Suppose first that the insertion positions are not weakly increasing. Then
    \(q_{k+1}<q_k\) for some \(k\). Immediately before inserting \(k+1\), the
    entry \(k\) is in position \(q_k\). After inserting \(k+1\), some entry
    \(y<k\) lies between \(k+1\) and \(k\). Thus \(k+1,y,k\) is a
    \(312\)-pattern.
    
    Conversely, suppose that the insertion positions are weakly increasing
    and that \(v\) contains a \(312\)-pattern with values \(x>z>y\)
    occurring in the order \(x,y,z\). When \(z\) is inserted, \(y\) is
    already present. Since \(y\) precedes \(z\) in \(v\), it must lie to
    the left of position \(q_z\) at that moment.
    For every \(k>z\), we have \(q_k\geq q_z\), so \(k\) is inserted
    to the right of \(y\).
    In particular, \(x\) is inserted to the right
    of \(y\), a contradiction.
\end{proof}

\begin{proposition}\label{prop:rbg-pattern-avoidance}
Let \(w\in B_n\).
Then \(L_{B_n}(w)\) is an rbg tuple if and only if \(w\) avoids the eight patterns in~\eqref{eq: 8 bad patterns}.
\end{proposition}

\begin{proof}
    \(\boldsymbol{(\Longrightarrow)}\) Suppose first that \(L_{B_n}(w)\) is an rbg tuple, and let \(r,t\), and
    \(p_k\) be as in Lemma~\ref{lem:rbg-insertion-characterization}.
    Thus, \(w\) satisfies conditions \ref{item:rbg-negative-subword}, \ref{item:rbg-positive-positions}, and \ref{item:rbg-first-positive} of Lemma~\ref{lem:rbg-insertion-characterization}.
    
    If \(r=0\), by condition~\ref{item:rbg-positive-positions} and Lemma~\ref{lem:312-insertion-positions}, \(w\) avoids \(312\), and since there are no negative entries, it avoids all the bad patterns.
    
    Suppose \(r>0\). By condition~\ref{item:rbg-negative-subword}, every negative entry has absolute value at most \(r\), whereas every positive entry has absolute value greater than \(r\).
    Thus \(w\) avoids
    \begin{equation*}
        \sbar21
        \qquad\text{and}\qquad
        1\sbar2.
    \end{equation*}
    From the description of the negative subword in condition~\ref{item:rbg-negative-subword}, it avoids
    \begin{equation*}
        \sbar3\sbar2\sbar1,
        \qquad
        \sbar2\sbar3\sbar1,
        \qquad
        \sbar2\sbar1\sbar3.
    \end{equation*}
    Hence \(w\) does.
    
    For \(k>r\), let \(q_k\) be the insertion position of \(k\) within the
    positive subword, and let \(b_k\) be the number of negative entries to
    the left of \(k\). We claim that both \((q_k)_{k>r}\) and
    \((b_k)_{k>r}\) are weakly increasing. Indeed, fix
    \(r<k<\ell\leq n\). By
    condition~\ref{item:rbg-positive-positions},
    \begin{equation*}
        p_k\leq p_{k+1}\leq\cdots\leq p_\ell.
    \end{equation*}
    Consequently, every entry preceding position \(p_k\) at step \(k\)
    still precedes the position in which \(\ell\) is inserted. Counting
    the positive and negative entries in this prefix gives
    \begin{equation*}
        q_k\leq q_\ell
        \qquad\text{and}\qquad
        b_k\leq b_\ell.
    \end{equation*}
    
    The monotonicity of the \(q_k\)'s and Lemma~\ref{lem:312-insertion-positions} show that the positive subword
    of \(w\) avoids \(312\), and hence so does \(w\). If \(w\) contained
    \(3\sbar12\), there would be positive entries \(\ell>k\) with
    \(\ell\) to the left of a negative entry and \(k\) to its right.
    This would give
    \begin{equation*}
        b_\ell<b_k,
    \end{equation*}
    contradicting \(b_k\leq b_\ell\). Thus \(w\) also avoids
    \(3\sbar12\).
    
    It remains to consider \(3\sbar2\sbar1\). In any occurrence of this
    pattern, the two negative entries form a pair whose absolute values
    decrease.
    By condition~\ref{item:rbg-negative-subword}, every such pair in the negative
    subword begins with \(\sbar r\). If \(t=r\), no such pair exists. If
    \(t<r\), condition~\ref{item:rbg-first-positive} gives
    \(p_{r+1}\geq p_r+1\), so \(r+1\) is inserted to the right of \(\sbar r\).
    Condition~\ref{item:rbg-positive-positions} then shows that every later
    positive entry is also inserted to the right of \(\sbar r\). Thus no
    positive entry can precede a decreasing pair of negative entries, and
    \(w\) avoids \(3\sbar2\sbar1\).

    \medskip
    \noindent \(\boldsymbol{(\Longleftarrow)}\) Conversely, suppose that \(w\) avoids the eight patterns.
    Let \(r:=\ell_0(w)\) and let \(p_k\) be the position where \(k\) or \(\sbar{k}\) is inserted in~\(w\).

    If \(r=0\), conditions~\ref{item:rbg-negative-subword} and~\ref{item:rbg-first-positive} are vacuous.
    Since there are no negative entries and \(w\) avoids \(312\), Lemma~\ref{lem:312-insertion-positions} implies condition~\ref{item:rbg-positive-positions}.
    By Lemma~\ref{lem:rbg-insertion-characterization}, \(L_{B_n}(w)\) is an rbg tuple.

    Suppose that \(r>0\).
    
    Avoidance of
    \(\sbar21\) and \(1\sbar2\) implies that every negative absolute value is
    smaller than every positive absolute value.
    Hence the negative entries have absolute values \(1,\ldots,r\).
    
    Let \(t\) be the position of \(\sbar r\) in the negative subword. Since the word at step \(r\) consists precisely of
    the negative entries \(\sbar1,\ldots,\sbar r\), and subsequent insertions are positive and do not change their relative order, we have \(t=p_r\).
    
    We claim that this subword is obtained from
    \begin{equation*}
        \sbar1\,\sbar2\,\cdots\,\sbar r
    \end{equation*}
    by moving \(\sbar r\) to position \(t\).
    
    Indeed, a descent in absolute value before \(\sbar r\) would produce
    an occurrence of \(\sbar2\sbar1\sbar3\), while a descent in absolute
    value after \(\sbar r\) would produce an occurrence of
    \(\sbar3\sbar2\sbar1\).
    Thus the entries on each side
    of \(\sbar r\) are increasing in absolute value. Moreover, if an entry
    before \(\sbar r\) had a larger absolute value than an entry after it,
    these three entries would form an occurrence of
    \(\sbar2\sbar3\sbar1\). Therefore every absolute value before
    \(\sbar r\) is smaller than every absolute value after it. Since the
    remaining absolute values are \(1,\ldots,r-1\), the claimed form follows.
    This proves condition~\ref{item:rbg-negative-subword}.
    
    For each \(k>r\), let \(q_k\) be the insertion position of \(k\) within
    the positive subword, and let \(b_k\) be the number of negative entries
    preceding the position in which \(k\) is inserted.
    Right after inserting \(k\), \(q_k-1\) positive entries and \(b_k\) negative entries appear at the left of \(k\).
    Hence
    \begin{equation*}
        p_k=q_k+b_k.
    \end{equation*}
    
    Avoidance of \(312\) and
    Lemma~\ref{lem:312-insertion-positions} imply
    \begin{equation*}
        q_{r+1}\leq q_{r+2}\leq\cdots\leq q_n.
    \end{equation*}
    
    Avoidance of \(3\sbar12\) implies
    \[
        b_{r+1}\leq b_{r+2}\leq\cdots\leq b_n.
    \]
    Indeed, otherwise \(b_x<b_y\) for some \(r<y<x\).
    Since all negative entries have already been inserted by step \(r\),
    \(b_x\) and \(b_y\) are also the numbers of negative entries to the
    left of \(x\) and \(y\), respectively, in \(w\).
    Thus some negative entry \(\sbar z\)
    lies to the right of \(x\) and to the left of \(y\).
    Since \(z\leq r<y<x\), the entries \(x,\sbar z,y\) form an occurrence of
    \(3\sbar12\), a contradiction.
    
    Therefore, \(p_k=q_k+b_k\) gives
    \begin{equation*}
        p_{r+1}\leq p_{r+2}\leq\cdots\leq p_n.
    \end{equation*}
    This proves condition~\ref{item:rbg-positive-positions}.
    
    Finally, suppose that \(0<r<n\) and \(t<r\). Then \(\sbar r\) precedes
    \(\sbar t\) in the negative subword.
    If a positive entry \(x\) occurred to the left of \(\sbar r\), then
    \(x,\sbar r,\sbar t\) would form an occurrence of
    \(3\sbar2\sbar1\). Thus every positive entry lies to the right of
    \(\sbar r\). At step \(r\), the entry \(\sbar r\) occupies position
    \(t\); therefore \(r+1\) must be inserted in a position strictly to
    its right, giving
    \begin{equation*}
        p_{r+1}\geq t+1.
    \end{equation*}
    All the conditions in
    Lemma~\ref{lem:rbg-insertion-characterization} now hold, so
    \(L_{B_n}(w)\) is an rbg tuple.
\end{proof}

\begin{proof}[Proof of Theorem~\ref{thm:equivalence of rbg codes}]
    The equivalence of~\ref{thm:equivalence of rbg codes-i} and
    \ref{thm:equivalence of rbg codes-ii} follows from
    Proposition~\ref{prop:kempf-rbg}, while the equivalence of
    \ref{thm:equivalence of rbg codes-ii} and
    \ref{thm:equivalence of rbg codes-iii} follows from
    Proposition~\ref{prop:rbg-pattern-avoidance}. Finally, since \(L_{B_n}\) is a
    bijection, Proposition~\ref{prop: counting rbg} gives
    \begin{equation*}
        \left\lvert\mathrm{Pr}(L_{B_n})\right\rvert
        =\lvert\mathcal C_n\rvert
        =C_n+C_{n+1}-1.\qedhere
    \end{equation*}
\end{proof}

For a set \(\mathcal B\) of signed permutation patterns, let
\begin{equation*}
    \operatorname{Av}_{B_n}(\mathcal B)
    :=
    \{w\in B_n:\text{\(w\) avoids every pattern in \(\mathcal B\)}\}.
\end{equation*}
Two sets \(\mathcal B\) and \(\mathcal B'\) of signed permutation
patterns are \emph{type~\(B\) Wilf equivalent} if
\begin{equation*}
    \bigl|\operatorname{Av}_{B_n}(\mathcal B)\bigr|
    =
    \bigl|\operatorname{Av}_{B_n}(\mathcal B')\bigr|
\end{equation*}
for every \(n\geq 1\).

We now introduce the two pattern sets that will be compared.
Let \(\mathcal B_8\) denote the set of eight patterns
in~\eqref{eq: 8 bad patterns}, and define
\begin{equation*}
    \mathcal B_{12}
    :=
    \{w_1w_2w_3\in B_3:w_1>w_2>w_3\}
    \cup
    \{
        \sbar{1}\sbar{2},
        \sbar{2}\sbar{1},
        \sbar{3}21,
        2\sbar{3}1
    \}.
\end{equation*}

\begin{corollary}
    The pattern sets \(\mathcal B_8\) and \(\mathcal B_{12}\) are
    type~\(B\) Wilf equivalent.
\end{corollary}

\begin{proof}
    By Theorem~\ref{thm:equivalence of rbg codes},
    \begin{equation*}
        \operatorname{Av}_{B_n}(\mathcal B_8)
        =
        \mathrm{Pr}(L_{B_n}),
    \end{equation*}
    and therefore
    \begin{equation*}
        \bigl|\operatorname{Av}_{B_n}(\mathcal B_8)\bigr|
        =
        C_n+C_{n+1}-1.
    \end{equation*}
    On the other hand, by~\cite[Corollary~5.7]{Ste97},
    \(\operatorname{Av}_{B_n}(\mathcal B_{12})\) is the set of fully
    commutative bottom elements of \(B_n\). By~\cite[Proposition~5.9]{Ste97},
    this set also has cardinality \(C_n+C_{n+1}-1\).
\end{proof}
\section{Palindromic Poincaré polynomials in
\texorpdfstring{$B_n$}{Bn}}

By Proposition~\ref{prop:kempf-rbg}, the type~\(B\) family
\(\Pi_q(\mathcal C_n)\) is precisely the set of Poincaré polynomials
arising from \(L_{B_n}\)-principal elements:
\begin{equation*}
    \{P_w(q):w\in\mathrm{Pr}(L_{B_n})\}
    =
    \Pi_q(\mathcal C_n).
\end{equation*}
To prove Theorem~\ref{thm:main-intro}, it remains to show that every palindromic Poincaré polynomial in~\(B_n\) belongs either to the type~\(B\) family or to the type~\(A\) family. For this, we use the following factorization rules of
Billey~\cite[Theorem~3.3]{Bil98}.
\begin{proposition}\label{prop:Billey_factorization_rules}
    Let \(z=z_1z_2\cdots z_k\in B_k\) be written in one-line notation, and suppose that \(z_d=\pm k\) and \(z_k=\pm e\), where \(1\leq d,e\leq k\).
    In each of the following five cases, one has
    \begin{equation*}
        P_z(q)=(1+q+\cdots+q^\mu)P_{z'}(q).
    \end{equation*}
    The integer \(\mu\), called the \emph{factorization parameter}, and the element \(z'\) are specified in each case.
    \begin{enumerate}[
        label=\textup{(R\arabic*)},
        ref=\textup{(R\arabic*)},
        leftmargin=*
        ]
        \item\label{rule:billey-1}
        If \(z_d=k>z_{d+1}>\cdots>z_k\), then
        \begin{equation*}
            \mu=k-d,
            \qquad
            z'=z s_d s_{d+1}\cdots s_{k-1}.
        \end{equation*}
        \item\label{rule:billey-2}
        If \(z_k=e>0\) and \(k=(z^{-1})_e>(z^{-1})_{e+1}>\cdots>(z^{-1})_k\), then
        \begin{equation*}
            \mu=k-e,
            \qquad
            z'=s_{k-1}s_{k-2}\cdots s_ez.
        \end{equation*}        
        \item\label{rule:billey-3}
        If \(z_d=\sbar k\), every entry of~\(z\) is negative, and the sequence
        obtained from \(z_1,\ldots,z_k\) by deleting \(z_d\) is strictly
        decreasing, then
        \begin{equation*}
            \mu=d+k-1,
            \qquad
            z'=z s_{d-1}s_{d-2}\cdots s_1s_0s_1\cdots s_{k-1}.
        \end{equation*}
        \item\label{rule:billey-4}
        If \(z_k=\sbar e\), every entry of~\(z\) is negative, and \( z_1>z_2>\cdots>z_{k-1}\), then 
        \begin{equation*}
            \mu=e+k-1, \qquad z'=s_{k-1}s_{k-2}\cdots s_1s_0s_1\cdots s_{e-1}z.
        \end{equation*}
        \item\label{rule:billey-5}
        If every entry of~\(z\) is positive except for \(z_d=\sbar k\), and \( z_1>z_2>\cdots>z_d\), then
        \begin{equation*}
            \mu=d,
            \qquad
            z'=z s_{d-1}s_{d-2}\cdots s_1s_0.
        \end{equation*}
    \end{enumerate}
    In \ref{rule:billey-1}--\ref{rule:billey-4}, the final entry of \(z'\) is~\(k\), and we identify \(z'\) with an element of \(B_{k-1}\) by deleting this entry.
    After this identification, the elements obtained in \ref{rule:billey-3} and \ref{rule:billey-4} are both equal to the longest element
    \[
        \sbar1\,\sbar2\cdots\sbar{k-1}
    \]
    of \(B_{k-1}\).
    In \ref{rule:billey-5}, the element \(z'\) is unsigned,
    so \(z'\in S_k\subset B_k\).

    We use the convention \(B_0=\{\mathrm{id}\}\) and interpret empty products and strings in the usual way.
    In particular, \ref{rule:billey-1} and \ref{rule:billey-2} both apply
    to the identity in \(B_k\), with \(d=e=k\) and factorization parameter
    \(0\), and produce the identity in \(B_{k-1}\).
\end{proposition}
We use these rules recursively.

\begin{definition}\label{def:complete-recursive-factorization}
    Let \(w\in B_n\).
    Starting with \(w\) at stage \(n\), construct a sequence of rule applications recursively as follows.
    Suppose that the element at stage \(i\) has been defined and is regarded as an element \(z\in B_k\).
    We call \(i\) the \emph{stage index} and~\(k\) the \emph{ambient index}.
    
    Then, choose one of \ref{rule:billey-1}--\ref{rule:billey-5} that applies to~\(z\).
    Let \(z'\) be the resulting element and \(\mu\) its factorization
    parameter, and set \(a_i:=\mu\). If \(i>1\), define \(z'\) to be
    the element at stage \(i-1\) and continue. If \(i=1\), stop after obtaining \(z'\) and recording \(a_1\).

    A sequence of rule applications obtained in this way is called a
    \emph{complete recursive factorization} of \(w\) if the element obtained after the application at stage \(1\) is the identity.
    The resulting tuple
    \[
        \mathbf a=(a_1,\ldots,a_n)
    \]
    is called the \emph{factorization tuple} associated with this
    complete recursive factorization, or simply a \emph{factorization
    tuple} of \(w\).
\end{definition}
Under each of \ref{rule:billey-1}--\ref{rule:billey-4}, the stage and ambient indices both decrease by one.
By contrast, \ref{rule:billey-5} decreases only the stage index.
Indeed, if it occurs at stage~\(k\), it is applied in \(B_k\) and produces an unsigned element of \(S_k\subset B_k\), which becomes the element at stage \(k-1\), still regarded as an element of \(B_k\).
Afterward, only \ref{rule:billey-1} and \ref{rule:billey-2} can occur, and both preserve unsignedness.
In particular, \ref{rule:billey-5} can occur at most once.
Thus the two indices agree until \ref{rule:billey-5} is applied, whereas afterward the ambient index is one greater than the stage index.

\begin{example}\label{ex:stage-and-ambient-indices}
    Let \(w=4321\in B_4\). Initially, both the stage and ambient
    indices are \(4\).
    Successive applications of
    \ref{rule:billey-1} give \(a_4=3\), \(a_3=2\), \(a_2=1\).
    The successive pairs (stage index, ambient index) are
    \[
        (4,4),\qquad(3,3),\qquad(2,2),\qquad(1,1).
    \]
    At stage \(1\), the current element is the identity.
    Applying \ref{rule:billey-1} with \(d=1\) gives \(a_1=0\) and produces the identity in \(B_0\).
    Thus the associated factorization tuple is \((0,1,2,3)\).

    By contrast, start with \(321\sbar4\in B_4\) at stage \(4\).
    Applying \ref{rule:billey-5} contributes \(a_4=4\) and produces
    the unsigned element \(4321\in B_4\). This element occurs at stage
    \(3\), so its stage index is \(3\), whereas its ambient index is
    still \(4\).
    The successive pairs of indices are
    \[
        (4,4),\qquad(3,4),\qquad(2,3),\qquad(1,2).
    \]
    The applications of \ref{rule:billey-1} at stages \(3,2,1\) give
    \(a_3=3\), \(a_2=2\), and \(a_1=1\).
    Hence the associated factorization tuple is \((1,2,3,4)\).
\end{example}

By~\cite[Theorem~4.2 and Lemma~4.4]{Bil98}, every rationally smooth
\(w\in B_n\) admits a complete recursive factorization. If
\(\mathbf a\) is any factorization tuple of \(w\), then
\[
    P_w(q)=\Pi_q(\mathbf a).
\]

\begin{lemma}\label{lem:billey-factor-tuple-rbg}
    Fix a complete recursive factorization of \(w\in B_n\) in which~\ref{rule:billey-5} never occurs, and let \(\mathbf a\) be its factorization tuple.
    Then \(\mathbf a\in\mathcal C_n\).
    Consequently, \(P_w(q)\in\Pi_q(\mathcal C_n)\).
\end{lemma}

\begin{proof}
    Since \ref{rule:billey-5} never occurs, at every stage \(k\), one of
    \ref{rule:billey-1}--\ref{rule:billey-4} is applied to an element of
    \(B_k\).
    Their formulas give \(0\leq a_k\leq2k-1\).
    
    It remains to verify the reflected bounded-growth condition.

    Fix \(2\leq k\leq n\) and suppose that
    \begin{equation*}
        a_{k-1}<2k-3.
    \end{equation*}
    If \ref{rule:billey-3} or \ref{rule:billey-4} were used at stage~\(k\),
    then the current element at stage \(k-1\) would be the longest element
    of \(B_{k-1}\), giving \(a_{k-1}=2k-3\), a contradiction. Thus
    \ref{rule:billey-1} or \ref{rule:billey-2} is used at stage~\(k\).
    Applying inversion throughout the factorization preserves its
    factorization parameters and exchanges \ref{rule:billey-1} with \ref{rule:billey-2}, and \ref{rule:billey-3} with \ref{rule:billey-4}.
    We may therefore assume that \ref{rule:billey-1} is used at stage \(k\).
    
    Let \(z\in B_k\) be the element at this stage and let \(d\) be the
    position of \(k\). Then
    \[
        z_d=k>z_{d+1}>\cdots>z_k,
        \qquad
        a_k=k-d.
    \]
    Let \(z'\in B_{k-1}\) be the element obtained by deleting \(k\)
    from \(z\).

    Suppose first that \ref{rule:billey-1} or
    \ref{rule:billey-2} is used at stage \(k-1\). Let \(b\) be the
    position of \(k-1\) in \(z'\) in the first case, and let
    \(b=z'_{k-1}>0\) in the second. In either case,
    \[
        a_{k-1}=k-1-b.
    \]
    We claim that \(b\leq d\).
    In the \ref{rule:billey-1} case, if \(k-1\) occurs after \(k\)
    in \(z\), then the decreasing-suffix condition forces
    \(z_{d+1}=k-1\). After deleting \(k\), the position of \(k-1\) is
    therefore \(b=d\). Otherwise, \(k-1\) occurs at some position \(r<d\),
    which is unchanged when \(k\) is deleted, so \(b=r<d\).
    In the \ref{rule:billey-2} case,
    the claim is immediate if \(d=k\). If \(d<k\), then \(z_k=b\), and the \(k-d-1\) entries
    \(z_{d+1},\ldots,z_{k-1}\) are distinct elements of
    \(\{b+1,\ldots,k-1\}\).
    Hence
    \[
        k-d-1\leq k-b-1,
    \]
    which again gives \(b\leq d\). Therefore
    \[
        a_k=k-d\leq k-b=a_{k-1}+1.
    \]
    Moreover, \(a_{k-1}\leq k-2\) and \(a_k\leq k-1\), so
    \[
        a_k\leq k-1\leq2k-3-a_{k-1}.
    \]

    Suppose instead that \ref{rule:billey-3} or
    \ref{rule:billey-4} is used at stage \(k-1\). Let \(b\) be the
    position of \(\sbar{k-1}\) in \(z'\) in the first case, and write
    \(z'_{k-1}=\sbar b\) in the second. In either case,
    \[
        a_{k-1}=k-2+b.
    \]
    The assumption \(a_{k-1}<2k-3\) therefore gives \(b\leq k-2\).

    In the \ref{rule:billey-3} case, \(\sbar{k-1}\) is the smallest
    entry of \(z'\) but is not its final entry. It cannot therefore
    occur in the decreasing suffix of \(z\) following \(k\), and hence
    \(d\geq b+1\). In the \ref{rule:billey-4} case,
    \(\sbar{k-1}\) occupies position \(k-2\) in \(z'\), and the same
    argument gives
    \[
        d\geq k-1\geq b+1.
    \]
    Thus, in either case,
    \[
        a_k=k-d
        \leq k-b-1
        =2k-3-a_{k-1}.
    \]
    Finally, \(a_{k-1}\geq k-1\) and \(a_k\leq k-1\), so
    \(a_k\leq a_{k-1}+1\).

    Hence the reflected bounded-growth condition holds for every~\(k\),
    and therefore \(\mathbf a\in\mathcal C_n\). The identity
    \(P_w(q)=\Pi_q(\mathbf a)\) gives the final assertion.
\end{proof}

The exclusion of \ref{rule:billey-5} is necessary.
Indeed, the second factorization in Example~\ref{ex:stage-and-ambient-indices} has factorization tuple
\[
    (1,2,3,4)\notin\mathcal C_4,
\]
since \(a_2=2<3\), and \(3=a_3>3-a_2=1\).

Since adjoining an initial zero does not change the \(q\)-integer
product, the same polynomial is represented in the type~\(A\) family by
\[
    (0,1,2,3,4)\in\widehat F_5\subseteq\mathcal C_5.
\]

The main application of Theorem~\ref{thm: iso posets An to Bn intro} is the following lemma, which gives the key dichotomy needed to prove Theorem~\ref{thm:main-intro}.
\begin{lemma}\label{lem:billey-rbg-or-bottom}
    Let \(w\in B_n\) be rationally smooth.
    Then
    \begin{equation*}
        P_w(q)\in \Pi_q(\mathcal C_n)
        \quad\text{or}\quad
        P_w(q)\in \{P_x(q) : x\in \mathrm{RSm}(A_n)\}.
    \end{equation*}
\end{lemma}
\begin{proof}
    By the preceding discussion, \(w\) admits a complete recursive
    factorization. If one can be chosen in which
    \ref{rule:billey-5} never occurs, then
    Lemma~\ref{lem:billey-factor-tuple-rbg} gives
    \[
        P_w(q)\in\Pi_q(\mathcal C_n),
    \]
    and the first alternative holds. We may therefore suppose that every
    complete recursive factorization of \(w\) uses
    \ref{rule:billey-5}.
    Choose one such factorization. Suppose that
    \ref{rule:billey-5} occurs at stage \(m\), and let \(y\in B_m\) be the
    current element.    
    
    We claim that only \ref{rule:billey-1} and~\ref{rule:billey-2} can occur before the application of \ref{rule:billey-5}.
    Indeed, if either \ref{rule:billey-3} or
    \ref{rule:billey-4} occurred at stage~\(k\) before
    \ref{rule:billey-5}, then it would produce the longest element of
    \(B_{k-1}\).
    This element admits a complete recursive factorization using only \ref{rule:billey-3} and \ref{rule:billey-4}.
    Replacing the remaining part of the chosen factorization by such a factorization would therefore produce a complete recursive factorization of~\(w\) in which \ref{rule:billey-5} never occurs, contradicting our assumption.
    
    Recall that \(\ell_0(v)\) is the number of negative entries of
    \(v\in B_k\). Rules \ref{rule:billey-1} and \ref{rule:billey-2}
    preserve \(\ell_0\): their formulas involve only generators \(s_i\)
    with \(i\geq1\), and in both cases the deleted entry~\(k\) is positive.
    The hypotheses of \ref{rule:billey-5} imply that \(y\) has exactly one negative entry, namely \(\sbar m\).
    Therefore,
    \begin{equation*}
        \ell_0(w)=\ell_0(y)=1.
    \end{equation*}
    
    By Proposition~\ref{prop: rex of x inter min zeros}, there exists
    \(x\in A_n\) such that
    \begin{equation*}
        w=\operatorname{bot}(x).
    \end{equation*}
    By Corollary~\ref{cor:intro An KL in Bn}%
    \ref{cor:intro An KL in Bn-poincare},
    we have \(P_x(q)=P_w(q)\).
    Since \(w\) is rationally smooth, \(P_w(q)\) is palindromic.
    Hence
    \(P_x(q)\) is palindromic, so \(x\in\mathrm{RSm}(A_n)\).
    It follows that
    \begin{equation*}
        P_w(q)\in\{P_x(q) : x\in\mathrm{RSm}(A_n)\}.\qedhere
    \end{equation*}
\end{proof}

\begin{proof}[Proof of Theorem~\ref{thm:main-intro}]
    Lemma~\ref{lem:billey-rbg-or-bottom} gives the inclusion from left
    to right.

    Conversely, let \(\mathbf a\in\mathcal C_n\) and set
    \(u=L_{B_n}^{-1}(\mathbf a)\). By
    Proposition~\ref{prop:kempf-rbg}, the element \(u\) is
    \(L_{B_n}\)-principal, and hence \(P_u(q)=\Pi_q(\mathbf a)\).
    Since \(\Pi_q(\mathbf a)\) is a product of \(q\)-integers, it is
    palindromic, so \(u\in\mathrm{RSm}(B_n)\).
    Thus every polynomial in
    the type~\(B\) family occurs in \(B_n\).

    Now let \(x\in\mathrm{RSm}(A_n)\). By
    Corollary~\ref{cor:intro An KL in Bn}%
    \ref{cor:intro An KL in Bn-poincare}, \(P_x(q)=P_{\operatorname{bot}(x)}(q)\).
    Hence \(\operatorname{bot}(x)\in\mathrm{RSm}(B_n)\), so every polynomial in the type~\(A\) family also occurs in~\(B_n\).
\end{proof}


\end{document}